\documentclass[11pt, reqno]{amsart}
\usepackage[dvipsnames]{xcolor}
\usepackage{mathtools}
\mathtoolsset{showonlyrefs}
\usepackage{etoolbox}

\usepackage{blindtext}
\usepackage{latexsym}
\usepackage{txfonts}
\usepackage{lmodern}
\usepackage[pagewise]{lineno}
\usepackage[utf8]{inputenc}
\usepackage{exscale}
\usepackage{amsmath}
\usepackage{amssymb}
\usepackage{amsfonts}
\usepackage{mathrsfs}
\usepackage{amsbsy}
\usepackage{relsize}
\usepackage[colorlinks, linkcolor=BrickRed, citecolor=blue, urlcolor=blue, hypertexnames=false]{hyperref}
\usepackage{comment}
\PassOptionsToPackage{reqno}{amsmath}
\usepackage{esint} % For \fint symbol
\usepackage{amssymb} % For \lesssim symbol
\usepackage{stmaryrd}
\usepackage{cite}
\newtheorem*{Acknow}{Acknowledgments}

\usepackage[a4paper,left=30mm,right=30mm,top=30mm,bottom=30mm,marginpar=20mm]{geometry}

\newtheorem{thm}{Theorem}[section]
\newtheorem{prop}{Proposition}[section]
\newtheorem{defi}{Definition}[section]
\newtheorem{lem}{Lemma}[section]

\newtheorem{rem}{Remark}[section]

\newcommand{\R}{\mathbb{R}}

\numberwithin{equation}{section}

\newcommand{\eps}{\epsilon}

\newcommand{\wto}{\rightharpoonup}

\makeatletter \@addtoreset{equation}{section} \makeatother

\newcounter{const}
\author[T. Gou, M. Majdoub \& T. Saanouni]{Tianxiang Gou, Mohamed Majdoub and Tarek Saanouni}
\address[T. Gou]{School of Mathematics and Statistics, Xi'an Jiaotong University, Xi'an 710049, Shaanxi, People's Republic of China.}
\email{\sl  {tianxiang.gou@xjtu.edu.cn}}
\address[M. Majdoub]{Department of Mathematics, College of Science, Imam Abdulrahman Bin Faisal University, P. O. Box 1982, Dammam, Saudi Arabia.\newline
Basic and Applied Scientific Research Center, Imam Abdulrahman Bin Faisal University, P.O. Box 1982, 31441, Dammam, Saudi Arabia.}
\email{\sl {mmajdoub@iau.edu.sa}}
\email{\sl {med.majdoub@gmail.com}}
\address[T. Saanouni]{Departement of Mathematics, College of Science, Qassim University, Buraydah, Kingdom of Saudi Arabia.}
\email{\sl {t.saanouni@qu.edu.sa}}



\title[Energy Critical INLS]{Radial 3D Focusing Energy Critical INLS equations with defocusing perturbation: Ground states, Scattering, and Blow-up}

\begin{document}

\begin{abstract}
We investigate the following inhomogeneous nonlinear Schr\"odinger equation in the radial regime, featuring a focusing energy-critical nonlinearity and a defocusing perturbation:
$$
 \textnormal{i} \partial_t u +\Delta u =|x|^{-a} |u|^{p-2} u - |x|^{-b} |u|^{4-2b}u \quad \mbox{in} \,\, \R_t \times \R_x^3,
$$
where $0<a$, $b<2$ and ${ 2+\frac{4-2a}{3}< p\leq 6-2a}$. 

First, we establish the existence and nonexistence of ground states, along with their quantitative properties. Subsequently, we analyze the dichotomy between scattering and blow-up for solutions with energy below the ground-state energy threshold.

An intriguing feature of this equation is the lack of scaling invariance, which arises from the competing effects of the inhomogeneous nonlinearities. Additionally, the presence of singular weights breaks translation invariance in the spatial variable, introducing further complexity to the analysis.

To the best of our knowledge, this work represents the first comprehensive study of the inhomogeneous nonlinear Schr\"odinger equation with a leading-order focusing energy-critical inhomogeneous nonlinearity and a defocusing perturbation. Our results provide new insights into the interplay between these competing nonlinearities and their influence on the dynamics of solutions.
\end{abstract}

\hrule
\maketitle
\hrule

\thispagestyle{empty}

\tableofcontents
\eject
%%%%%%%%%%%%%%%%%%%%%%%%%%%%%%%%%%%%%%%%%%%%%%%%%%%%%%%%%%%%%%%%%%%%%%%%%%%%%%%%%%
\section{Introduction}\label{S1}
%%%%%%%%%%%%%%%%%%%%%%%%%%%%%%%%%%%%%%%%%%%%%%%%%%%%%%%%%%%%%%

In this paper, we consider the following { focusing energy critical inhomogeneous} nonlinear Schr\"odinger equation with perturbed { defocusing} inhomogeneous nonlinearity,
\begin{align} \label{equ}
\textnormal{i} \partial_t u +\Delta u =|x|^{-a} |u|^{p-2} u - |x|^{-b} |u|^{4-2b}u \quad \mbox{in} \,\, \R_t \times \R_x^3,
\end{align}
where $u(t,x)$ is a complex-valued function { defined} on { the} spacetime $\R_t\times\R_x^3$, $0<a, b<2$, and ${ 2+\frac{4-2a}{3}< p\leq 6-2a}$. Equation \eqref{equ} provides a versatile framework for describing a wide range of physical phenomena { in plasma physics, nonlinear optics, and Bose-Einstein condensation}. For deeper insights and further interpretations, we refer the readers to  \cite{BPVT, KA} and the works cited therein.

Given its Hamiltonian structure, the solutions { to the Cauchy problem for} \eqref{equ} exhibit the  conservation of { the} energy { defined by}
\begin{equation}\label{Ener}
E(u(t)):=\frac 12 \int_{\R^3} |\nabla u(t,x)|^2 \,dx +\frac{1}{p} \int_{\R^3} |x|^{-a} |u(t,x)|^{p} \,dx - \frac{1}{6-2b} \int_{\R^3} |x|^{-b} |u(t,x)|^{6-2b} \,dx.
\end{equation}
Furthermore, the solutions { to the Cauchy problem for} \eqref{equ} also preserve the mass { defined by}
\begin{equation}  \label{Mass}
M(u(t))=\int_{\R^3} |u(t,x)|^2 \,dx.
\end{equation}
 
\begin{comment}
\vspace{0.5cm}
\hrule 
{\color{red} 
\begin{itemize}
\item It seems that we need the assumption
$$
6-2b>a+\frac{3(p-2)}{2} \Longleftrightarrow 3a>2b.
$$
\textcolor{blue}{We agree with $6-2b>a+\frac{3(p-2)}{2}$ which is contained in our previous paper but we don't see how to obtain $3a>2b$. It seems $a>b$!}
\item The mass critical exponent $p=2+\frac{4-2a}{3}$ is covered. It is unknown to me if this will have an influence on the discussion of scattering. To my knowledge, for the study of scattering of solutions to NLS with combined power type nonlinearities, perturbed term is often assumed to be the mass supercritical case.\\
\textcolor{blue}{For the scattering, we need  $2+\frac{4-2a}{3}<p<6-2a$.}
\item When $p=6-2a$, the perturbed term $|x|^{-a} |u|^{p-2} u$ is energy critical. Will this bring out essential difficulties for scattering part? When $p=6-2a$, the problem becomes double critical.\\
\textcolor{blue}{We suppose $2+\frac{4-2a}{3}<p<6-2a$.}
\end{itemize}
}
\hrule 
\vspace{0.5cm}
\end{comment}

{ Equation} \eqref{equ} represents a specific case within a wider category of inhomogeneous nonlinear Schr\"odinger equations characterized by double { inhomogeneous} nonlinearities. These equations can be generally formulated as:
\begin{equation}\label{INLS-2}
\textnormal{i} \partial_t u +\Delta u =\lambda\,|x|^{-a} |u|^{p-2} u +\mu\, |x|^{-b} |u|^{q-2}u \quad \mbox{in} \,\, \R_t \times \R_x^N,
\end{equation}
where $a, b\geq 0$, $\lambda, \mu \in\R$, and $p, q>2$. The aforementioned equation generalized classical models like the nonlinear Schr\"odinger { equations} (NLS) and the inhomogeneous nonlinear Schr\"odinger { equations} (INLS). { Such equations describe the propagation of laser beams in certain types of plasma media and nonlinear optics, see for example \cite{LT, Gi} and the references therein}. 

Specifically, when $a=b=\mu=0$, { then \eqref{INLS-2} becomes the standard nonlinear Schr\"odinger equation}
\begin{equation} \label{NLS}
 \textnormal{i} \partial_t u +\Delta u =\lambda\, |u|^{p-2} u  \quad \mbox{in} \,\, \R_t \times \R_x^N.
\end{equation}
{ Equation \eqref{NLS} has been extensively studied in the last decades. Local well-posedness of solutions to the Cauchy problem for \eqref{NLS} in the energy space was first established by Ginibre and Velo in \cite{GV}. The existence of finite-time blow-up solutions was proved by Glassey in \cite{Gl}. Later on, there exists a great deal of literature devoted to the consideration of dynamics of solutions.
Dynamical properties of blow-up solutions in the mass critical case were investigated in \cite{M1, MR1, MR2, MR3, MT, We}. The scattering versus blow-up of solutions below the ground state energy level in the mass supercritical case were considered in \cite{ADM, dm, DM, DHR, FXC, HR, KM, KV}.}

In the case where $\mu=0$ and $a>0$, { then \eqref{INLS-2} becomes the inhomogeneous nonlinear Schr\"odinger equation}
\begin{equation} \label{INLS}
\textnormal{i} \partial_t u +\Delta u =\lambda\,|x|^{-a} |u|^{p-2} u  \quad \mbox{in} \,\, \R_t \times \R_x^N.
\end{equation}
It is worth mentioning that there { exists} a large amount of research on the initial value problem associated to \eqref{NLS}. 
{ By using the abstract theory introduced in \cite{Ca}, Genoud and Stuart \cite{GS} initially established the local well-posedness of solutions to the Cauchy problem for \eqref{INLS} in the energy space for { $0<a<2$ and $2<p<2+\frac{2-a}{(N-2)_{+}}$}\footnote{Hereafter we use the notation $\kappa_+ :=\max(\kappa,0)$ with the convention $0^{-1}=\infty$.}. Subsequently, the existence and dynamics of blow-up solutions were considered in \cite{CG', D, Far, Genoud2012, Gou}. The sharp thresholds for scattering versus blow-up of solutions below the ground state energy level were derived in \cite{Campos, cc, CFGM, DK, FG, MMZ} for the energy sub-critical case and in \cite{chl, CL, GM} for the energy critical case. We also refer the readers to the monographs \cite{Ca, Bour, Fib, LP, SS, Ta}, where a very extensive overview on the most established results with respect to solutions to \eqref{NLS} and \eqref{INLS} is covered. }

{  When $a=b=0$, then \eqref{INLS-2} reduces to the following nonlinear Schr\"odinger equation with combined power nonlinearities,
\begin{align} \label{dnls}
\textnormal{i} \partial_t u +\Delta u =\lambda\ |u|^{p-2} u +\mu\,  |u|^{q-2}u \quad \mbox{in} \,\, \R_t \times \R_x^N.
\end{align}
Equation \eqref{dnls} has recently received substantial attention after the pioneering work due to Tao, Visan, and Zhang \cite{TVZ}, where the global well-posedness and scattering of solutions were studied across different regimes.} Inspired by \cite{TVZ}, there {{exists} a large number of} significant { literature} devoted to the study of solutions to \eqref{dnls}. The scattering and { blow-up} of solutions { to \eqref{dnls} in the inter-critical case} was investigated by Bellazzini et al. \cite{BDF} for a defocusing perturbation and by Xie \cite{Xie} for a focusing perturbation. { Moreover,} the scattering versus blow-up dichotomy under the ground state threshold was considered in \cite{KOPV, KRV} for a defocusing energy critical { perturbation in three space dimensions}. The same { topic} was { discussed} for a focusing energy-critical perturbed nonlinearity in \cite{MXZ}. { These results were} extended to { the} lower dimensions in \cite{CMZ} and to four space dimensions in \cite{MTZ}. 

{ Inspired by the study performed previously, it is interested to consider solutions to nonlinear Schr\"odinger equations with combined inhomogeneous nonlinearities}. Recently, { in \cite{GMS}, we revealed a range of properties of solutions to \eqref{INLS-2}, where $N \geq 1$, $0<a,b< \min \{2, N\}$, $\lambda>0$, $\mu<0$, and $ 2<p<q<\frac{2N-2b}{(N-2)_+}$. The investigation included the existence/nonexistence, symmetry, decay, uniqueness, non-degeneracy, and instability of ground states}.
 
{ We also} established the scattering below the ground state energy threshold by employing Tao’s scattering criterion and Dodson-Murphy’s Virial/Morawetz inequalities in the non-radial regime. Furthermore, we provided an upper bound on the blow-up rate. 
Let us also mention \cite{GC}, where the authors investigated  { solutions to  }\eqref{INLS-2} under the conditions $2<N<6$, $0<a,b<\min\{2, \frac{6-N}{2}\}$, $\lambda\mu\neq 0$, $2<p<\frac{2N-2a}{N-2}$, and $q=\frac{2N-2b}{N-2}$. They { achieved} the global well-posedness of solutions in the energy space 
and blow-up phenomenon { of solutions} for initial data in $\Sigma:=H^1\cap L^2(|x|^2 dx)$ { with} negative energy. { It is worth pointing out that the research carried out in \cite{GC} does not overlap the one in the current paper.}

{ As an extension of} the results outlined in \cite{GMS}, we { are going to further} explore a scenario involving a focusing energy-critical nonlinearity perturbed by a { defocusing nonlinearity} in three spatial dimensions, where $N=3$, $\lambda=1, \mu=-1$, $0<a,b<2$, { $2<p \leq 6-2a$} and $q=6-2b$. { The principal aim of the paper is to establish the existence/nonexistence and quantitative properties of ground states and to prove the scattering versus blowup below the ground-state energy threshold.
}

The primary goal of this paper is to explore standing wave solutions to \eqref{equ}. A standing wave solution takes the form
$$
u(t,x)=e^{\textnormal{i} \omega t} \phi(x), \quad \omega \in \R,
$$
where $\phi\in H^1(\R^3)$ satisfies the associated elliptic equation
\begin{align} \label{equ1}
-\Delta \phi + \omega \phi = |x|^{-b} |\phi|^{4-2b}\phi- |x|^{-a} |\phi|^{p-2} \phi \quad \mbox{in} \,\, \R^3.
\end{align}

\begin{defi} (Ground state)
{ A} solution $Q \in H^1(\R^3)$ to \eqref{equ1} is called a ground state if it is non-trivial and minimizes the associated action functional
$$
S_{\omega}(\phi):=\frac 12 \int_{\R^3} |\nabla \phi|^2 \,dx + \frac {\omega}{2} \int_{\R^3} |\phi|^2 \,dx +\frac{1}{p} \int_{\R^3} |x|^{-a} |\phi|^{p_1} \,dx - \frac{1}{6-2b} \int_{\R^3} |x|^{-b} |\phi|^{6-2b} \,dx,
$$
over all non-trivial { solutions to} \eqref{equ1}. That is,
$$
S_\omega(Q) = \inf \Big\{ S_\omega(\phi) \ : \ \phi \in H^1 \backslash \{0\} \; \text{ solves } \eqref{equ1} \Big\}.
$$
\end{defi}

The first result of the paper concerning the existence and nonexistence of solutions to \eqref{equ1} reads as follows.

\begin{thm} \label{existence}
{ Let $0<a,b<2$. Then the following assertions hold true.}
\begin{itemize}
\item [$(\textnormal{i})$] If $2<p \leq 6-2a$ and $\omega>0$, then there exists no solutions in $H^1(\R^3)$ to \eqref{equ1}.
\item [$(\textnormal{ii})$] If $2<p<6-2a$ and $\omega=0$, then there exists no solutions in ${H}^1(\R^3)$ to \eqref{equ1}. 
\item [$(\textnormal{iii})$] If $p=6-2a$, $b<a$ and $\omega=0$, then there exist positive ground states in $\dot{H}^1(\R^3)$ to \eqref{equ1}.
\item [$(\textnormal{iv})$] If $2<p \leq 6-2a$ and $\omega<0$, then there exists no radially symmetric nonnegative solutions in $H^1(\R^3)$ to \eqref{equ1}.
\end{itemize}
\end{thm}

{ To prove the nonexistence of solutions to \eqref{equ1} for $\omega \geq 0$, one needs to make use of Pohozaev identity satisfied by solutions to \eqref{equ1}, see Lemma \ref{ph}. While, to prove the nonexistence of solutions to \eqref{equ1} for $\omega<0$, one needs to take advantage of \cite[Lemma 4.2]{MS}. Finally, to establish the existence of solutions to \eqref{equ1} for $\omega=0$, we shall introduce the following minimization problem,
\begin{align} \label{min0}
m_0:=\inf_{\phi \in P} E(\phi),
\end{align}
where
$$
P:=\left\{\dot{H}^1(\R^3)\backslash\{0\} : K(\phi)=0 \right\},
$$
\begin{align*}
K(\phi):=\int_{\R^3} |\nabla \phi|^2 \,dx+\frac{3(p-2)+2a}{2p}\int_{\R^3}|x|^{-a}|\phi|^{p}\, dx-\int_{\R^3}|x|^{-b}|\phi|^{6-2b}\, dx.
\end{align*}
Here $K(\phi)=0$ is the so-called Pohozaev identity related to \eqref{equ1} and $P$ is the Pohozaev manifold, which is indeed a natural constraint. It is straightforward to find that any minimizer to \eqref{min0} is a ground state to \eqref{equ1}. To detect the existence of minimizers, the essential difficulty is to check the compactness of sequences due to the presence of double energy critical exponents $p=6-2a$ and $q=6-2b$. Our discussion is carried out in the spirit of the Lions concentration compactness principle.
}

\begin{rem}
{\rm We { now} briefly discuss some quantitative properties of the solutions to \eqref{equ1} derived in Theorem \ref{existence} for $p=6-2a$, $b<a$ and $\omega=0$ .
\begin{itemize}
\item [$(\textnormal{i})$] Using the polarization arguments developed in \cite{BWW} and proceeding as the proof of \cite[Theorem 1.1]{GMS}, we can similarly get that any ground state $\phi$ is radially symmetric and decreasing.

\item [$(\textnormal{ii})$] In view of \cite[Theorem 1.3]{GMS}, we are able to derive that any positive, radially symmetric and decreasing solution $\phi$ enjoys the optimal decay $\phi(x) \sim |x|^{-1}$ as $|x| \to \infty$. This clearly leads to $\phi \not \in L^2(\R^3)$.
\item [$(\textnormal{iii})$] Observe that 
$$
\frac{4-a}{8-2a} -\frac{2a+4(6-2a)}{(6-2b)(8-2a)} +\frac{b}{6-2b}=0.
$$
As an application of \cite[Theorem 1.3]{GMS}, we then know that there exists at most one positive, radially symmetric and decreasing solution.
\end{itemize}
}
\end{rem}

\begin{thm} \label{nonexistence}
Let $0<b<a<2$ and $2<p < 6-2a$. Define
\begin{align} \label{minn}
m:=\inf_{\phi \in P} E(\phi), \quad P:=\left\{ \phi \in H^1(\R^3) \backslash \{0\} : K(\phi)=0\right\}.
\end{align}
Then there exists no minimizers to \eqref{minn}. Moreover, there holds that
\begin{align} \label{cm}
m= E^c(Q),
\end{align}
where 
{ $$
E^c(u):=\frac12 \int_{\R^3 }|\nabla u|^2 \,dx -\frac1{6-2b}\int_{\R^3}|x|^{-b}|u|^{6-2b}\,dx,
$$ }
$Q \in \dot{H}^1(\R^3)$ is the ground state to the equation
\begin{align}\label{defq}
-\Delta Q =|x|^{-b} |Q|^{4-2b}Q, \quad Q(x)=\left(1+\frac{|x|^{4-2b}}{3-b}\right)^{\frac{1}{b-2}}.
\end{align}
\end{thm}

{ The nonexistence of minimizers is a direct consequence of Theorem \ref{existence}, because any minimizer corresponds to a solution to \eqref{equ1}. To detect \eqref{cm}, we need to invoke Lemma \ref{clem1}, which gives an alternative variational characterization of the ground state energy level $m$.}

\begin{rem}
{\rm 
The aforementioned ground state {$ Q \in \dot H^1(\R^3)$} serves as a minimizer { of} the { minimization} problem
\begin{equation} \label{min2}
\frac1{C_*} :=\inf_{ u\in\dot H^1(\R^3) \backslash\{0\}}\frac{\|\nabla u\|}{\left(\int_{\R^3}|x|^{-b}|u|^{6-2b}\,dx\right)^\frac{1}{6-2b}},
\end{equation}
where the notation $\|\cdot\|$ represents the $L^2-$norm { throughout the paper}. One can easily find that 
\begin{equation}\label{min'}
\frac1{C_*}= \frac{\|\nabla{Q}\|}{\left(\int_{\R^3}|x|^{-b}|{Q}|^{6-2b}\,dx\right)^\frac1{6-2b}}=\|\nabla{Q}\|^\frac{2-b}{3-b}.
\end{equation}
}
\end{rem}

Next, we will analyze the dynamic behavior of solutions to the Cauchy problem associated with Equation (\ref{equ}). We will initiate this by outlining the well-posedness in $H^1(\mathbb{R}^3)$.
\begin{prop}
{ Let} $0 < a, b < 2$ and $2 < p \leq 6 - 2a$. Then, for any $u_0 \in H^1(\mathbb{R}^3)$, there exist { $T_{\text{max}} > 0$ and a unique solution $u \in C([0, T_{\text{max}}), H^1(\mathbb{R}^3))$ to the Cauchy problem for \eqref{equ} with $u(0)=u_0$} satisfying the conservation of { the mass and the energy}. That is, for any $t \in [0, T_{\text{max}})$,
$$
M(u(t))=M(u_0), \quad E(u(t))=E(u_0),
$$
where the mass and the energy are defined by \eqref{Mass} and \eqref{Ener}, respectively.

Furthermore, the solution mapping $u_0 \mapsto u$ is continuous from $H^1(\mathbb{R}^3)$ to $C([0, T_{\text{max}}), H^1)$. It also holds that either $T_{\text{max}} < +\infty$ or $\displaystyle\lim_{t \to T^-_{\text{max}}} \|\nabla u(t)\| = +\infty.$
\end{prop}
{ It is noteworthy that the proof of the aforementioned proposition can be completed by following closely the steps outlined in \cite{guz}, where the well-posedness of solutions to equations with single inhomogeneous nonlinearity was tackled}. Consequently, the  details are omitted for brevity.

In accordance with \cite{ps}, we introduce the Payne-Sattinger sets as follows,
\begin{align}
\mathcal{K}^-:=\left\{\phi \in H^1_{rad}(\R^3) \backslash \{0\} : E(\phi)<m,\, K(\phi)  < 0 \right\},\label{k--}\\
\mathcal{K}^+:=\left\{\phi \in H^1_{rad}(\R^3) \backslash \{0\} : E(\phi)<m,\, K(\phi) \geq 0 \right\}, \label{k++}
\end{align}
{ where $m \in \R$ is defined by \eqref{minn}.} We first investigate the scattering of solutions to the Cauchy problem { for} \eqref{equ} when initial data { belong to} the set $\mathcal{K}^+$. Our main result in this context can be stated as follows.

\begin{thm} \label{scattering}
Let ${ 0<a<1}$, $0<b<\frac43$, { $a<b$} and $2+\frac{4-2a}{3} < p< 6-2a$. { Then there exists a unique global solution $u \in C(\R, H^1_{rad}(\R^3))$ to the Cauchy problem for \eqref{equ} with $u(0) \in \mathcal{K}^+$}. Furthermore, { if $p\geq 4$}, { then} the solution $u$ scatters in $H^1_{rad}(\R^3)$.
\end{thm}

%\begin{rem}
{ To achieve the global well-posedness of solutions in Theorem \ref{scattering}, we are inspired by \cite{XZ}. First we apply perturbation arguments to obtain a good local well-posedness of solutions to the Cauchy problem for \eqref{equ} with initial data belonging to $\mathcal{K}^+$, see Lemma \ref{prt}. Then, by utilizing a coercivity property of solutions, see Lemma \ref{crcv}, we are able to conclude the global well-posedness of solutions.
}

The methods employed to prove the scattering { in \cite{chl, GM}} are grounded in the concentration-compactness-rigidity technique pioneered by Kenig and Merle \cite{KM} in the context of the energy-critical NLS equation. { However, to establish the scattering of solutions in Theorem \ref{scattering}, we alternatively apply the more recent approach introduced by Dodson and Murphy \cite{dm}, building upon Tao's scattering criterion \cite{tao} and Virial/Morawetz estimates.
}

\begin{rem}
{\em
 { The condition $a<b$ is applied to derive the variational characterization of the ground state energy level $m$ defined by \eqref{min'}, see Lemma \ref{clem1}, by which Lemma \ref{k} holds true. Note that Lemma \ref{k} is essential to prove the global existence of solutions, which indeed guarantees the condition \eqref{as5} in Lemma \ref{crcv} is valid. Furthermore, the condition $0<a<1$ is due to the application of \cite[Lemma 3.4]{guz}, which is adapted to estimate the Strichartz norm related to the energy sub-critical nonlinearity, see Lemma \ref{prt}. The condition $0<b<\frac 43$, which appears in \cite[Theorem 1]{chl}, is required to control the Strichartz norm related to the energy critical nonlinearity, see Lemmas \ref{lm1} and \ref{prt}.
These restrictions seem hard to circumvent in our arguments.}
}
\end{rem}

\begin{rem}
{\em The condition $p\geq 4$ { is imposed on the proof of the scattering criteria, see Lemma \ref{crt}, which} is technical and related to our methodology. We believe that it could be potentially relaxed.}
\end{rem}

%%%%%%%%%%%%%%%%%%%%%%%%%%%%%%%%%%%%%%%%%%%%%%%%%%%%%%%%%%%%%%%%%%%%%%%%%%%%

Lastly, we will investigate the blow-up phenomenon. The ensuing result illustrates that the blow-up of solutions to the Cauchy problem for \eqref{equ} occurs when the initial data belong to the set $\mathcal{K}^-$.

\begin{thm} \label{blowup}
Let $0<b<a<2$ and $2+\frac{4-2a}{3} < p< 6-2a$. Let $u\in C([0, T_{max}), H^1_{rad}(\R^3))$ be the solution to the Cauchy problem for \eqref{equ} with $u(0) \in \mathcal{K}^-$. Then $u$ blows up in finite time.
\end{thm}

{ To establish Theorem \ref{blowup}, we first need to establish variational characterizations of ground state energy level $m$ given by \eqref{min'}, see Lemma \ref{clem2}, where the assumption $b<a$ is imposed. Then, by analyzing the evolution of localized virial quantity defined by
$$
V_R(t):=\int_{\R^3} \psi_R(x) |u(t, x)|^2 \,dx,
$$
where $R>0$ and $\psi_R : \R^3 \to \R$ is a proper cut-off function, see Lemma \ref{mrwz1}, we attain the desirable conclusion. This completes the proof.}

\begin{rem}
 {\rm We anticipate { that blow-up of solutions occurs as well} for non-radial data subject to specific constraints on $p$. This can be accomplished { with minor changes by following closely the investigation conducted} in \cite{GMS}.}   
\end{rem} 

\begin{rem}
{\rm The upper bound on { blow-up rate} of solutions can be directly derived from \cite[Theorem 1.6]{GMS} within { the} radial context.}
\end{rem}

The outline of the article is as follows. In Section \ref{S2}, we commence by introducing key notations, revisiting standard identities, and presenting a series of valuable results. Section \ref{S3} is dedicated to exploring the existence or nonexistence of solutions to \eqref{equ1}, where we provide the proof of Theorem \ref{existence}. In Section \ref{S4}, we  examine the local { and global} existence of \eqref{equ} within the energy space $C_T(H^1)$. In Section \ref{S5}, we address the energy scattering of { solutions to} \eqref{equ} { and offer} the proof of Theorem \ref{scattering}. Section \ref{S6} will conclude our discussion by examining the blow-up of solutions to the Cauchy problem for \eqref{equ} and presenting the proofs of Theorems \ref{nonexistence} and \ref{blowup}.

Throughout the { paper}, the symbol $C$ will represent various positive constants that are inessential to the analysis and may vary from line to line. We use the notation  $X\lesssim Y$ to indicate the estimate  $X\leq CY$ for certain constant $C>0$.

For convenience, we shall use the notation $\|\cdot\|$ to denote the $L^2$ norm throughout the paper.

%%%%%%%%%%%%%%%%%%%%%%%%%%%%%%%%%%
\section{Preliminaries} \label{S2}
%%%%%%%%%%%%%%%%%%%%%%%%%%%%%%%%%%%%%%%%%%%%%%%%%%%%%%%%%%%%%%%%%%%%%%%%%%%%%%%%%%%%%%%%%%%%%%%%%%%%%%%%%%%%%%%%%%%

In this section, we will introduce some tools and auxiliary results, { which are crucial to establish} our main findings. We will start by presenting the well-known Gagliardo-Nirenberg inequality from \cite{Far, Genoud2012}.

\begin{lem} \label{gn}
Let  $0<\kappa<2$ and $2<r \leq 6-2\kappa$. Then there exists $C>0$ such that, for any $u \in H^1(\R^3)$,
\begin{align} \label{GN}
\int_{\R^3} |x|^{-\kappa} |u|^{r} \, dx \leq C\|\nabla u\|^{\frac{3}{2}r + \kappa-3}\, \|u\|^{3-\kappa -\frac{r}{2}}.
\end{align}
\end{lem}

We recall the Caffarelli-Kohn-Nirenberg weighted interpolation inequality { from} \cite{sgw,csl}.

\begin{lem}\label{ckn}
{ Let $1 \leq p, q<\infty$, $a \leq b$ and $b>\frac 3q, a>\frac 3{p}$. Assume that 
$$
b-a-1=3\left(\frac1p-\frac1q \right).
$$ 
Then there exists $C>0$ such that, for any $u \in C^{\infty}_0(\R^3)$,
$$
\||\cdot |^{-b} u\|_{L^q(\R^3)}\leq C\||\cdot|^{-a}\nabla  u\|_{L^p(\R^3)}.
$$
}
\end{lem}

We also recall the Strauss inequality for radial functions { from}\cite{BL, Strauss}.

\begin{lem} \label{rlem}
There exists $C>0$ such that, for any $u \in H_{rad}^1(\R^3)$,
\begin{align} \label{rineq}
{ |u(x)| \leq C |x|^{-1}  \|u\|^{\frac 12}\|\nabla u\|^{\frac 12}, \quad |x| \geq R},
\end{align}
{ where $R>0$}.
\end{lem}

Let us also present the well-known { Sobolev} inequality in $\dot{H}^1(\R^3)$, which can be regarded as a special case for $\kappa=b$ and $r=6-2b$ in Lemma \ref{gn}.

\begin{lem} \label{sem}
Let $0<b<2$. { Then} there exists $S_b>0$ such that, for any $u \in \dot{H}^1(\R^3)$,
\begin{align} 
S_b\left(\int_{\R^3} |x|^{-b}|u|^{6-2b} \,dx \right)^{\frac {1}{6-2b}} \leq \|\nabla u\|,
\end{align}
where the equality holds if and only if $u=Q$ and $Q$ is given by \eqref{defq}.
\end{lem}

\begin{lem} \label{ph}
Let $u \in H^1(\R^3)$ be a nontrivial solution to \eqref{equ1}. Then $K(u)=0$, i.e
$$
\int_{\R^3} |\nabla u|^2 \,dx+\frac{3(p-2)+2a}{2p}\int_{\R^3}|x|^{-a}|u|^{p}\, dx=\int_{\R^3}|x|^{-b}|u|^{6-2b}\, dx.
$$
\end{lem}
\begin{proof}
To establish this result, one can adopt a similar approach to the proof outlined in \cite[Lemma 2.4]{GMS}. For brevity, we will skip the proof here.
\end{proof}

\begin{lem} \label{maximum}
Let $0<b<a<2$ and $2+\frac{2(2-a)}{3}<p < { 6-2a}$. Then{,  for any $u \in {H}^1(\R^3) \backslash \{0\}$}, there exists a unique $\lambda_u>0$ such that { $K(u_{\lambda_u})=0$} and 
\begin{align} \label{max1-1}
\max_{\lambda>0} E(u_\lambda)=E(u_{\lambda_u}),
\end{align}
{ where $u_{\lambda}:=\lambda^{\frac 32}u(\lambda \cdot)$ for $\lambda>0$.}
Moreover, if $K(u) < 0$, then $0<\lambda_u < 1$. Additionally, the function $\lambda \mapsto E(u_{\lambda})$ is concave on $[\lambda_u, +\infty)$.
\end{lem}
\begin{proof}
Observe that
\begin{align} \label{max1}
\begin{split}
\frac{d}{d\lambda} E(u_{\lambda})&=\lambda\int_{\R^3} |\nabla u|^2 \,dx+\frac{3(p-2)+2a}{2p}{\lambda}^{\frac 3 2 (p-2)+a-1}\int_{\R^3}|x|^{-a}|u|^{p} \, dx\\ 
&\quad-{\lambda}^{5-2b}\int_{\R^3}|x|^{-b}|u|^{6-2b} \, dx \\
&=\frac {1} {\lambda} K(u_{\lambda}).
\end{split}
\end{align}
In addition, we note that
\begin{align} \label{max2} \nonumber
\frac{d^2}{d{\lambda}^2} E(u_{\lambda})&=\int_{\R^3} |\nabla u|^2 \,dx+\frac{(3(p-2)+2a)((3(p-2)-2(1-a))}{4p}{\lambda}^{\frac{3}{2}(p-2)+a-2}\int_{\R^3}|x|^{-a}|u|^{p} \, dx \\
& \quad -(5-2b){\lambda}^{4-2b}\int_{\R^3}|x|^{-b}|u|^{6-2b} \, dx.
\end{align}
{ Since $0<b<a<2$ and $2+\frac{2(2-a)}{3}<p < 6-2a$, by \eqref{max1} and \eqref{max2}, then there exists a unique  $\lambda_u>0$ such that $K(u_{\lambda_u})=0$}. Moreover, { there holds that}
$$
\frac{d}{d\lambda} E(u_\lambda)>0\,\, \mbox{for} \,\, 0<\lambda<\lambda_u, \quad \frac{d}{d\lambda} E(u_\lambda)<0 \,\, \mbox{for} \,\, {\lambda >\lambda_u}.
$$ 
Hence {\eqref{max1} holds true}. Furthermore, { we are able to conclude} that $0 < \lambda_u < 1$ when $K(u) < 0$. Additionally, it is straightforward to observe the existence of a unique $0 <\widetilde{\lambda}_{u} \leq \lambda_u$ such that
$$
\frac{d^2}{d\lambda^2} E(u_{\lambda})\mid_{\lambda=\tilde{\lambda}_u}=0.
$$
{ Meanwhile, there holds that}
$$
\frac{d^2}{d\lambda^2} E(u_{\lambda})>0 \,\, \mbox{for} \,\, 0<\lambda<\widetilde{\lambda}_u, \quad \frac{d^2}{d\lambda^2} E(u_{\lambda})<0 \,\, \mbox{for} \,\, \lambda>\widetilde{\lambda}_u.
$$
This clearly implies that the function $\lambda \mapsto E(u_{\lambda})$ is concave on $[\lambda_u, +\infty)$. { This completes the proof}.
\end{proof}
Subsequently, our attention shifts towards scattering. { We} will assemble { some} requisite notations and tools to prepare the proof of Theorem \ref{scattering}. We begin with the so-called Strichrtz estimates.
\begin{defi}
Let ${s \in  [-1,1]}$. A pair of real numbers $(q, r)$ is called $s$-admissible if
$$
{ 2 \leq q, r \leq \infty,  \quad \frac{2}{q}+\frac{3}{r}=\frac{3}{2}-s}.
$$
When $s = 0$, we refer to the pair as admissible.
\end{defi}

Let $\Lambda_s$ denote the set of $s$-admissible pairs, meaning that
$$
\Lambda_s:=\Big\{(q, r) : (q, r) \,\, \mbox{is} \,\, s\mbox{-admissible} \Big\}.
$$ 
{ Let us} introduce the following Strichartz norm,
$$
\|u\|_{\Lambda^s(I)}:=\sup_{(q ,r) \in \Lambda_s} \|u\|_{L^q(I, L^r)}.
$$
Similarly, we define
$$
\Lambda_{-s}:=\left\{(q, r) : (q, r) \,\, \mbox{is} \,\, (-s)\mbox{-admissible} \right\}, \qquad \|u\|_{\Lambda^{-s'}(I)}:=\inf_{(q, r) \in \Lambda_{-s}}\|u\|_{L^{q'}(I, L^{r'})},
$$
where $(q', r')$ is the conjugate exponent pair of $(q ,r)$. For $s=0$, we will denote $\Lambda^s$ { by} $\Lambda$ and $\Lambda^{-s'}$ { by} $\Lambda'$. From now on, when referring to an interval $I \subset \mathbb{R}$, we { define} the spaces 
\begin{align} \label{S}
{ S(I):=L^{10}(I\times\R^3)}
\end{align}
$$
W(I):=L^{10}(I,L^{\frac{30}{13}}); 
$$
$$
\nabla W(I):=L^{10}(I,\dot W^{1,\frac{30}{13}}); 
$$
$$
W_{b}(I):=L^\frac{10(1+{b})}{1+3{b}}\left(I,L^{\frac{30(1+b)}{13+9b}}\right).
$$
{ Obviously, we find that} $W_0(I) = W(I)$. By the Sobolev embedding, we have that
\begin{align}
\nabla W(I)\hookrightarrow S(I).\label{inj}
\end{align}
Furthermore, { it is straightforward to check that} the following pairs are admissible,
\begin{align}
\left(10,\frac{30}{13}\right),\quad \left(\frac{10(1+{b})}{1+3{b}},\frac{30(1+b)}{13+9b}\right).
\end{align}

\begin{lem} \label{str} (\cite{Fo, KT})
{ Let $0 \leq s \leq 1$ and $I \subset \mathbb{R}$ be a time interval. Then the following assertions are true}.
\begin{itemize}
\item[$(\textnormal{i})$] $\displaystyle \|e^{\textnormal{i} t\Delta} f\|_{\Lambda^s(I)} \lesssim \|f\|_{\dot{H}^s}$;
\item[$(\textnormal{ii})$] $\displaystyle \left\|\int_0^t e^{\textnormal{i} (t-s)\Delta} g(\cdot, s) \,ds\right\|_{\Lambda^s(I)}+ \left\|\int_{\R} e^{\textnormal{i} (t-s)\Delta} g(\cdot, s) \,ds\right\|_{\Lambda^s(I)} \leq \|g\|_{\Lambda^{-s'}(I)}$;
\item[$(\textnormal{iii})$] $\displaystyle \left\|\int_I e^{\textnormal{i} (t-s)\Delta} g(\cdot, s) \,ds\right\|_{\Lambda^s(\R)} \leq \|g\|_{\Lambda^{-s'}(I)}$.
\end{itemize}
\end{lem}

\begin{rem}
{\rm { In the discussion of scattering}}, the restriction $N=3$ is due to the use of the { following} estimate from \cite[Lemma 3.1]{tao},
\begin{equation}\label{fretao}
\|e^{\textnormal{i}\cdot\Delta}u\|_{L^4(\R,L^\infty(\R^3))}\lesssim \|\nabla u\|.
\end{equation}
\end{rem}

We are now in a position to address a variational characterization of the ground state energy level $m$. Define

\begin{align*}
I(u):=\frac{3(p-2)-2(2-a)}{6(p-2)+4a}\int_{\R^3} |\nabla u|^2 \,dx + \frac{2(6-2b)-3(p-2)-2a}{(3(p-2)+2a)(6-2b)} \int_{\R^3} |x|^{-b}|u|^{6-2b} \,dx.
\end{align*}

\begin{lem} \label{clem1}
{ Let $0<b<a<2$ and $2+\frac{4-2a}{3}<p < 6-2a$}. There holds that
\begin{align*}
m&=\inf \left\{ I(u) : u \in H^1(\R^3) \backslash \{0\}, K(u) \leq 0\right\} \\
&=\inf \left\{ I(u) : u \in H^1(\R^3) \backslash \{0\}, K(u) < 0\right\}.
\end{align*}
\end{lem}
\begin{proof}
It is clear to see that
\begin{align} \label{c1-2}
\begin{split}
m&=\inf \left\{E(u) : u \in H^1(\R^3) \backslash \{0\}, K(u) = 0\right\} \\
& \geq \inf \left\{I(u) : u \in H^1(\R^3) \backslash \{0\}, K(u) \leq 0\right\}.
\end{split}
\end{align}
On the other hand, for any $K(u)<0$, from Lemma \ref{maximum}, there exists $0<\lambda_u<1$ such that $K(u_{\lambda_u})=0$. Therefore, we have that 
$$
E(u_{\lambda_u})=I(u_{\lambda_u})+\frac{2}{3(p-2)+2a}K(u_{\lambda_u})=I(u_{\lambda_u})<I(u).
$$ 
It then follows that
\begin{align} \label{c2}
\begin{split}
 m&=\inf \left\{E(u) : u \in H^1(\R^3) \backslash \{0\}, K(u) = 0\right\} \\
& \leq \inf \left\{I(u) : u \in H^1(\R^3) \backslash \{0\}, K(u) < 0\right\}.
\end{split}
\end{align}
Combining \eqref{c1-2} and \eqref{c2}, we then know that
\begin{align} \label{i1}
\begin{split}
&\inf \left\{ I(u) : u \in H^1(\R^3) \backslash \{0\}, K(u) \leq 0\right\} \\
& \leq m \leq \inf \left\{ I(u) : u \in H^1(\R^3) \backslash \{0\}, K(u) < 0\right\}.
\end{split}
\end{align}
If $K(u) \leq 0$, then $K(u_{\lambda}) <0$ for any $\lambda>1$ by Lemma \ref{maximum}. Observe that $I(u_{\lambda}) \to I(u)$ as $\lambda \to 1^+$. Therefore, we conclude that
\begin{align*}
&\inf \left\{ I(u) : u \in H^1(\R^3) \backslash \{0\}, K(u) < 0\right\} \\
& \leq \inf \left\{ I(u) : u \in H^1(\R^3) \backslash \{0\}, K(u) \leq 0\right\}.
\end{align*}
Coming back to \eqref{i1}, we then have the desired conclusion and the proof is completed.
\end{proof}

In what follows, we are going to provide an equivalent form of the Payne-Sattinger sets { defined by} \eqref{k--} and \eqref{k++}. For this, we introduce the following functionals, {
$$
E^c(u):=\frac12\int_{\R^3} |\nabla u|^2 \,dx-\frac{1}{6-2b}\int_{\R^3}|x|^{-b}|u|^{6-2b}\,dx,
$$
$$
K^c(u):=\int_{\R^3} |\nabla u|^2 \,dx-\int_{\R^3}|x|^{-b}|u|^{6-2b}\,dx.
$$
}
\begin{lem} \label{k}
{ Let $0<a<b<2$ and $2+\frac{2(2-a)}{3}<p < { 6-2a}$. Then} the sets defined { by} \eqref{k--} and \eqref{k++} can be represented { respectively} as
\begin{align} 
\mathcal{K}^+&=\left\{u \in H^1(\R^3) \backslash \{0\} : E(u)<m, \|\nabla u\|<\|\nabla Q\|  \right\},
\label{k+}\\
\mathcal{K}^-&=\left\{\phi \in H^1(\R^3) \backslash \{0\} : E(u)<m, \|\nabla u\|>\|\nabla Q\|  \right\}.\label{k-}
\end{align}
\end{lem}
\begin{proof}
{ Let us first demonstrate that \eqref{k+} hold true.} Let $u \in H^1(\R^3) \backslash \{0\}$ { be} such that $E(u)<m$ and $\|\nabla u\|<\|\nabla Q\|$.  We shall prove that $K(u)>0$. {Define a function on ${ [0, +\infty)}$ by $f(y):=y-C_*^{6-2b}y^{3-b}$, where $C_*>0$ is defined by \eqref{min2}}. Owing to \eqref{min2}, we infer that
\begin{align}
{ K(u) \geq \|\nabla u\|^2-\int_{\R^3}|x|^{-b}|u|^{6-2b}\,dx \geq \|\nabla u\|^2-\left(C_*\|\nabla u\|\right)^{6-2b}=f(\|\nabla u\|^2).
}
\end{align}
Furthermore, the zeros of $f$ are at zero and $y_0:=C_*^{-\frac{6-2b}{2-b}}=\|\nabla Q\|^2$ as per \eqref{min'}. The behavior of $f$ { clearly} indicates that $f(y) > 0$ on $(0,y_0)$. Therefore, { we find that} $\|\nabla u\| < \|\nabla Q\|$ implies that $K(u) > 0$. { It then follows that
\begin{align} \label{k+1}
\left\{u \in H^1(\R^3) \backslash \{0\} : E(u)<m, \|\nabla u\|<\|\nabla Q\|  \right\} \subset \mathcal{K}^+.
\end{align}
}
On the other hand, assuming { that} $E(u) < m$ and $K(u) > 0$, we aim to demonstrate that $\|\nabla u\| < \|\nabla Q\|$. 
{ Define a scaling of $u$ by} $u_\nu := \nu u(\nu^2 \cdot)$ { for} $\nu > 0$. 
{ It is simple to compute that
$$
\|\nabla u_\nu\|=\|\nabla u\|,
$$
$$
\int_{\R^3}|x|^{-b}|u_\nu|^{6-2b}\,dx= \int_{\R^3}|x|^{-b}|u|^{6-2b}\,dx,
$$
$$
\int_{\R^3}|x|^{-a}|u_\nu|^{p}\,dx=\nu^{p+2a-6} \int_{\R^3}|x|^{-a}|u|^{p}\,dx.
$$
}
It { then} follows that
{
$$
I(u_\nu)=I(u),\quad   K^c(u_\nu):=K^c(u).
$$
\begin{align}
K(u_\nu)=\|\nabla u\|^2+\frac{(3(p-2)+2a)\nu^{p+2a-6}}{2p}\int_{\R^3}|x|^{-a}|u|^{p}\,dx-\int_{\R^3}|x|^{-b}|u|^{6-2b}\,dx.
\end{align}
}
{ Supposing that $K^c(u) < 0$, then we find $\nu >1$ such that $K(u_\nu) = 0$, because of $K(u)>0$ and $p<6-2b$. This along with Lemma \ref{clem1} leads to
\begin{align*} 
E(u_\nu)=E(u_{\nu})-\frac{2}{3(p-2)+2a} K(u_{\nu})&=I(u_{\nu}) \\
&\geq m =\inf\left\{I(v) :  v  \in H^1(\R^3) \backslash \{0\}, K(v) \leq 0\right\}.
\end{align*}
}
This { obviously} contradicts the { fact} that $E(u_\nu) < E(u) < m$ for $\nu > 1$. { Then we derive that} $K^c(u) \geq 0$. 
Consequently, by \eqref{cm}, there holds that
\begin{align*}
\left (\frac12-\frac{1}{6-2b} \right)\|\nabla u\|^2 \leq E^c(u)
< E(u) <m=\left (\frac12-\frac{1}{6-2b} \right)\|\nabla Q\|^2.   
\end{align*}
{Hence we obtain the desirable conclusion, which means that
\begin{align} \label{k+2}
\mathcal{K}^+ \subset \left\{u \in H^1(\R^3) \backslash \{0\} : E(u)<m, \|\nabla u\|<\|\nabla Q\|  \right\}.
\end{align} 
}
This achieves \eqref{k+} by combining \eqref{k+1} and \eqref{k+2}. The proof of \eqref{k-} follows similarly.
\end{proof}

%%%%%%%%%%%%%%%%%%%%%%%%%%%%%%%%%%%%%%%%%%%%%%%%%%%%%%%%%%%%%%%%%%%%%%%%%%%%%%%%%%%%%%%%%%%%%%%%%%%%%%%%%%%%%%
\section{Existence/Nonexistence of Ground States}
\label{S3}
%%%%%%%%%%%%%%%%%%%%%%%%%%%%%%%%%%%%%%%%%%%%%%%%%%%%%%%%%%%%%%%%%%%%%%%%%%%%
In this section, we consider the existence/nonexistence of solutions to \eqref{equ1} and give the proof of Theorem \ref{existence}. Let us first introduce the following minimization problem,
\begin{align} \label{min}
m_0=\inf_{u \in P} E(u),
\end{align}
where 
$$
P:=\Big\{ u \in \dot{H}^1(\R^3) \backslash \{0\} : K(\phi)=0\Big\}.
$$

\begin{proof}[Proof of Theorem \ref{existence}] To begin with, we shall prove the nonexistence of solutions for $\omega > 0$. Let $u \in H^1(\R^3)$ be a solution to \eqref{equ1}, then
\begin{align} \label{ph1}
\int_{\R^3} |\nabla u|^2 \,dx+ \omega \int_{\R^3} |u|^2 \,dx + \int_{\R^3}|x|^{-a}|u|^{p}\, dx=\int_{\R^3}|x|^{-b}|u|^{6-2b}\, dx.
\end{align}
On the other hand, by Lemma \ref{ph}, then
\begin{align} \label{ph2}
\int_{\R^3} |\nabla u|^2 \,dx+\frac{3(p-2)+2a}{2p}\int_{\R^3}|x|^{-a}|u|^{p}\, dx=\int_{\R^3}|x|^{-b}|u|^{6-2b}\, dx.
\end{align}
Combining \eqref{ph1} and \eqref{ph2}, we then see that
\begin{align} \label{none1}
\omega \int_{\R^3} |u|^2 \,dx + \left(1-\frac{3(p-2)+2a}{2p}\right)\int_{\R^3}|x|^{-a}|u|^{p}\, dx=0.
\end{align}
Since $2<p \leq 6-2a$, then
$$
0<\frac{3(p-2)+2a}{2p} \leq 1
$$
Therefore, from \eqref{none1}, we get the nonexistence result for $\omega>0$. Similarly, from \eqref{none1}, we are able to derive the nonexistence of solutions for $\omega=0$. Next we shall consider the nonexistence of radially symmetric solutions to \eqref{equ1} for $\omega<0$. Let $u \in H^1(\R^3)$ be a nonnegative solution to \eqref{equ1} for $\omega<0$. By Lemma \ref{rlem}, then
$$
-\Delta u=\left(-\omega + |x|^{-b} |u|^{4-2b}u-|x|^{-a} |u|^{p-2} \right)u \geq 0, \quad |x| \geq R,
$$
where the constant $R>0$ is large. It then follows from \cite[Lemma 4.2]{MS} that $u=0$. 

Now we are going to establish the existence of non-negative ground states in $\dot{H}^1(\R^3)$ to \eqref{equ1} for $\omega=0$ and $p=6-2a$. For this, we shall make use of the minimization problem \eqref{min}. Let $\{w_n\} \subset P$ be a minimizing sequence to \eqref{min}, i.e. $E(w_n)=m_0+o_n(1)$. In view of Lemma \ref{sem}, then
$$
\int_{\R^3} |\nabla w_n|^2 \,dx+\int_{\R^3}|x|^{-a}|w_n|^{6-2a}\, dx=\int_{\R^3}|x|^{-b}|w_n|^{6-2b}\, dx \leq S_b^{2b-6} \left(\int_{\R^3} |\nabla w_n|^2 \, dx\right)^{3-b}.
$$
This immediately shows that 
\begin{align} \label{bdd1}
\int_{\R^3} |\nabla w_n|^2 \,dx \geq S_b^{\frac{6-2b}{2-b}}.
\end{align}
In addition, since $K(w_n)=0$, then 
\begin{align} \label{bddh}
\begin{split}
E(w_n)&=E(w_n)-\frac{1}{6-2b}K(w_n) \\
&=\left(\frac 12 - \frac{1}{6-2b}\right) \int_{\R^3}|\nabla w_n|^2\,dx+ \left(\frac{1}{6-2a}-\frac{1}{6-2b}\right) \int_{\R^3} |x|^{-a}|w_n|^{6-2a} \,dx.
\end{split}
\end{align}
Invoking \eqref{bdd1} and the assumption that $b<a<2$, we then know that $m_0>0$. It then follows from \eqref{bddh} that $\{w_n\}$ is bounded in $\dot{H}^1(\R^3)$. In the spirit of the well-known Ekeland's variational principle, then there exists a bounded Palais-Smale sequence $\{u_n\} \subset \dot{H}^1(\R^3)$ for $E$ at the level $m_0>0$. Since $\{u_n\}$ is bounded in $\dot{H}^1(\R^3)$, then there exists $u \in \dot{H}(\R^3)$ such that $u_n \wto u$ in $\dot{H}^1(\R^3)$ as $n \to \infty$. In addition, since $\{u_n\}$ is a Palais-Samle sequence for $E$, 
then
\begin{align} \label{eque1}
-\Delta u_n=|x|^{-b} |u_n|^{4-2b}u_n-|x|^{-a} |u_n|^{4-2a} u_n +o_n(1)  \quad \mbox{in} \,\, \R^3,
\end{align}
It then follows that $u \in \dot{H}^1(\R^3)$ solves the equation
\begin{align} \label{eque2}
-\Delta u=|x|^{-b} |u|^{4-2b}u-|x|^{-a} |u|^{4-2a} u  \quad \mbox{in} \,\, \R^3.
\end{align}
Define $v_n:=u_n-u$, then {
\begin{align}\label{l10}
\int_{\R^3} |\nabla v_n|^2 \,dx=\int_{\R^3}  |\nabla u_n|^2  \,dx-\int_{\R^3}  |\nabla u|^2  \,dx+o_n(1),
\end{align}
\begin{align}\label{l11}
\int_{\R^3} |x|^{-a}|v_n|^{6-2a} \,dx=\int_{\R^3} |x|^{-a}|u_n|^{6-2a} \,dx-\int_{\R^3} |x|^{-a}|u|^{6-2a} \,dx+o_n(1),
\end{align}
\begin{align}\label{l111}
\int_{\R^3} |x|^{-a}|v_n|^{6-2b} \,dx=\int_{\R^3} |x|^{-a}|u_n|^{6-2b} \,dx-\int_{\R^3} |x|^{-a}|u|^{6-2b} \,dx+o_n(1).
\end{align}
}
This indicates that
\begin{align}\label{l1}
E(v_n)=E(u_n)-E(u)+o_n(1)=m_0-E(u)+o_n(1).
\end{align}
Observe that $u_n \in \dot{H}^1(\R^3)$ solves \eqref{eque1}, $u \in \dot{H}^1(\R^3)$ solves \eqref{eque2}, by \eqref{l10}, \eqref{l11} and \eqref{l111}, then 
\begin{align} \label{l2}
\langle E'(v_n), v_n \rangle &=\int_{\R^3} |\nabla v_n|^2 \,dx +\int_{\R^3} |x|^{-a}|v_n|^{6-2a} \,dx-\int_{\R^3} |x|^{-b}|v_n|^{6-2b} \,dx=o_n(1).
\end{align}

If $\|\nabla v_n\|_2=o_n(1)$ or $u \neq 0$, then $u \in \dot{H}^1(\R^3)$ is a nontrivial solution to \eqref{eque2}. In view of Lemma \ref{ph}, we then have that $K(u)=0$ and $m_0 \leq E(u)$. Now we may assume that $\|\nabla v_n\|_2 \neq o_n(1)$ and $u=0$. Let $\chi_R \in C_0^{\infty}(\R^3)$ be a cut-off function such that $\chi_R(x)=1$ for $R \leq |x| \leq 2R$, $\chi_R(x) =0$ for $|x| > 5R/2$ or $|x|<R/2$, $0 \leq \chi_R(x) \leq 1 $ and ${ |\nabla \chi_R(x)| \lesssim 1}$ for $x \in \R^3$, where $R>1$ is a constant.
 
In light of H\"older's inequality, then
\begin{align*}
\int_{R/2 \leq |x| <5R/2} |x|^{-a}|u_n|^{6-2a} \,dx &\leq \left(\int_{R/2 \leq |x| <5R/2} |x|^{-a \theta} \,dx\right)^{\frac {1}{\theta}}\left(\int_{R/2 \leq |x| <5R/2} |u_n|^{\frac{\theta(6-2a)}{\theta-1}} \,dx \right)^{\frac {\theta-1}{\theta}} \\
& \lesssim R^{\frac{3-a \theta}{\theta}} \left(\int_{R/2 \leq |x| <5R/2} |u_n|^{\frac{\theta(6-2a)}{\theta-1}} \,dx \right)^{\frac {\theta-1}{\theta}},
\end{align*}
where 
$$
\theta >\frac{3}{a}>1, \quad 2<\frac{\theta(6-2a)}{\theta-1} <6.
$$
Since $u_n \wto 0$ in $\dot{H}^1(\R^3)$ as $n \to \infty$, then $u_n \to 0$ in $L^q_{loc}(\R^3)$ as $n \to \infty$ for $2 \leq q<6$. Hence there holds that
\begin{align} \label{la}
\int_{R/2 \leq |x| <5R/2} |x|^{-a}|u_n|^{6-2a} \,dx=o_n(1).
\end{align}
Similarly, there holds that
\begin{align} \label{lb}
\int_{R/2 \leq |x| <5R/2} |x|^{-b}|u_n|^{6-2b} \,dx=o_n(1).
\end{align}
Since $\langle E'(u_n), \chi_R^2 u_n \rangle =o_n(1)$, by \eqref{la} and \eqref{lb}, then
\begin{align} \label{nes}
o_n(1)=\langle E'(u_n), \chi_R^2 u_n \rangle= \int_{\R^3} |\chi_R \nabla u_n|^2 \,dx + \int_{\R^3} \nabla u_n \cdot \left(\nabla \chi_R^2\right) u_n \,dx +o_n(1).
\end{align} 
Utilizing H\"older's inequality and the fact that $\|\nabla u_n\|_2 \lesssim 1$, we see that
\begin{align*}
\left|\int_{\R^3} \nabla u_n \cdot \left(\nabla \chi_R^2\right) u_n \,dx \right| 
&\lesssim \left(\int_{\R^3} |\nabla u_n|^2|\nabla \chi_R|^2\,dx \right)^{\frac 12}\left(\int_{\R^2}|\chi_R|^2 |u_n|^2 \,dx\right)^{\frac 12} \\
&\lesssim \left(\int_{\R^3}|\chi_R|^{2\gamma} \,dx \right)^{\frac {1}{\gamma}} \left(\int_{R/2 \leq |x| \leq 5R/2} |u_n|^{\frac{2\gamma}{\gamma-1}} \,dx\right)^{\frac {\gamma-1}{\gamma}},
\end{align*}
where 
$$
\gamma>1, \quad 2<\frac{2\gamma}{\gamma-1}<6.
$$ 
Hence there holds that
\begin{align} \label{nes1}
\left|\int_{\R^3} \nabla u_n \cdot \left(\nabla \chi_R^2\right) u_n \,dx \right| =o_n(1).
\end{align}
Consequently, from \eqref{nes} and \eqref{nes1}, we obtain that
\begin{align} \label{nl}
\int_{R \leq |x| \leq 2R} |\nabla u_n|^2 \,dx=o_n(1).
\end{align}

Let $\eta_R \in C_0^{\infty}(\R^3)$ be a cut-off function such that $\eta_R(x)=1$ for $|x| \leq R$, $\eta_R(x)=0$ for $|x| \geq 2R$ and $0 \leq \eta_R(x) \leq 1$ for $x \in \R^3$. Therefore, by \eqref{nl}, we get that
$$
\int_{\R^3} \nabla u_n \cdot \nabla (u_n \eta_R) \,dx =\int_{|x|<R} |\nabla u_n|^2 \,dx +o_n(1).
$$
In addition, by \eqref{la} and \eqref{lb}, we know that
\begin{align} \label{b1}
\int_{\R^3} |x|^{-a}|u_n|^{6-2a} \eta_R\,dx=\int_{|x| <R} |x|^{-a}|u_n|^{6-2a} \,dx +o_n(1),
\end{align}
\begin{align} \label{b2}
\int_{\R^3} |x|^{-b}|u_n|^{6-2b} \eta_R \,dx=\int_{|x| <R} |x|^{-b}|u_n|^{6-2b} \,dx +o_n(1).
\end{align}
Since $\langle E'(u_n), \eta_R u_n\rangle=o_n(1)$, by \eqref{nl}, \eqref{b1} and \eqref{b2}, then
\begin{align} \label{le}
\int_{|x|<R} |\nabla u_n|^2 \,dx+\int_{|x| <R} |x|^{-a}|u_n|^{6-2a} \,dx=\int_{|x| <R} |x|^{-b}|u_n|^{6-2b} \,dx+o_n(1).
\end{align}
Observe by Lemma \ref{sem} that
$$
\left(\int_{\R^3} |x|^{-b}|\eta_R u_n|^{6-2b} \,dx \right)^{\frac{1}{6-2b}}\leq S_b^{-1} \left(\int_{\R^3} |\nabla (\eta_R u_n)|^2 \,dx\right)^{\frac 12}.
$$

Define
$$ 
\mu_R:=\lim_{n \to \infty} \int_{|x| <R} |\nabla u_n|^2 \,dx.
$$
$$
\nu_{b, R}:=\lim_{n \to \infty}\int_{|x| <R} |x|^{-b}|u_n|^{6-2b} \,dx,
$$
Taking into account \eqref{le}, we then obtain that
$$
\nu_{b, R}^{\frac{1}{6-2b}} \leq S_b^{-1} \mu_R^{\frac 12} \leq S_b^{-1}\nu_{b, R}^{\frac 12}.
$$
It clearly implies that $\nu_{b, R}=0$ or $\nu_{b, R} \geq S_b^{{(6-2b)}/{(2-b)}}$. 

{ Since $m_0>0$ and $K(u_n)=o_n(1)$, then
$$
\nu_{b, \infty}:=\lim_{n \to \infty}\int_{\R^3}|x|^{-b}|u_n|^{6-2b}\, dx>0.
$$ 
It then shows that $\nu_{b, R}>0$ for $R>1$ large enough. Then we get that, for some $0<\nu_0<S_b^{{(6-2b)}/{(2-b)}}$, there exists $R_n > 0$ such that
\begin{align*} 
\int_{|x|<R_n}|x|^{-b}|u_n|^{6-2b}\, dx= \nu_0.
\end{align*}
}
Define
$$
\widetilde{u}_n(x):=\left(\frac{R_n}{R}\right)^{\frac 12} u_n\left(\frac{R_n}{R}x\right), \quad x \in \R^3.
$$
It is straightforward to compute that
\begin{align} \label{c11}
\int_{|x|<R}|x|^{-b}|\widetilde{u}_n|^{6-2b} \,dx=\nu_0.
\end{align}
In addition, we find that
$$
\int_{\R^3}|\nabla \widetilde{u}_n|^2 \,dx=\int_{\R^3}|\nabla u_n|^2 \,dx,
$$
$$
\int_{\R^3}|x|^{-a}|\widetilde{u}_n|^{6-2a} \,dx=\int_{\R^3}|x|^{-a}|u_n|^{6-2a} \,dx,
$$
$$
 \int_{\R^3}|x|^{-b}|\widetilde{u}_n|^{6-2b} \,dx=\int_{\R^3}|x|^{-b}|u_n|^{6-2b} \,dx.
$$
Since $\{u_n\} \subset \dot{H}^1(\R^3)$ is a bounded Palais-Smale sequence for $E$ at the level $m_0$, then $\{\widetilde{u}_n\} \subset \dot{H}^1(\R^3)$ is a also bounded Palais-Smale sequence for $E$ at the level $m_0$. Replacing the role of $\{u_n\}$ by $\{\widetilde{u}_n\}$, we can derive as previously that there exists a solution $\widetilde{u} \in \dot{H}^1(\R^3)$ such that $K(\widetilde{u})=0$ and $m_0 \leq E(\widetilde{u})$. Otherwise, there holds that
$$
\lim_{n \to \infty}\int_{|x| <R} |x|^{-b}|\widetilde{u}_n|^{6-2b} \,dx=0 \,\, \mbox{or} \,\, \lim_{n \to \infty}\int_{|x| <R} |x|^{-b}|\widetilde{u}_n|^{6-2b} \,dx \geq S_b^{{(6-2b)}/{(2-b)}}.
$$
However, this is impossible by \eqref{c11}. Consequently, from the discussion above, we can conclude that there exists a nontrivial solution $u \in \dot{H}^1(\R^3)$ such that $K(u)=0$ and $m_0 \leq E(u)$. If $E(u)=m_0$, then the proof is completed. If not, then we assume that $E(u)>m_0$. Observe that
\begin{align*}
E(v_n)+o_n(1)&=E(v_n)-\frac{1}{6-2b} \langle E'(v_n), v_n \rangle \\
&=\left(\frac 12 - \frac{1}{6-2b}\right) \int_{\R^3}|\nabla v_n|^2 \,dx+ \left(\frac{1}{6-2a}-\frac{1}{6-2b}\right) \int_{\R^3} |x|^{-a}|v_n|^{6-2a} \,dx \geq 0.
\end{align*}
This along with \eqref{l1} implies that $E(u) \leq m_0$, which clearly contradicts with the assumption that $E(u)>m_0$. It then follows that $E(u)=m_0$. 

Observe that if $K(u)=0$, then $K(|u|) \leq 0$. Hence, by Lemma \ref{maximum}, there exists $0<\lambda_{|u|}<1$ such that $K(u_{\lambda_{|u|}})=0$. In addition, we have that
\begin{align*} 
E(|u|_{\lambda_{|u|}})&=E(|u|_{\lambda_{|u|}})-\frac{1}{6-2b}K(|u|_{\lambda_{|u|}}) \\
&=\left(\frac 12 - \frac{1}{6-2b}\right) \int_{\R^3}|\nabla(|u|_{\lambda_{|u|}})|^2 \,dx+ \left(\frac{1}{6-2a}-\frac{1}{6-2b}\right) \int_{\R^3} |x|^{-a}|u_{\lambda_{|u|}}|^{6-2a} \,dx \\
&<\left(\frac 12 - \frac{1}{6-2b}\right) \int_{\R^3}|\nabla u|^2\,dx+ \left(\frac{1}{6-2a}-\frac{1}{6-2b}\right) \int_{\R^3} |x|^{-a}|u|^{6-2a} \,dx \\
&=E(u)-\frac{1}{6-2b}K(u)=E(u).
\end{align*}
Thereby, we can assume that $u$ is nonnegative. From the maximum principle, we then know that $u$ is positive. Thus, the proof is completed.
\end{proof}

%%%%%%%%%%%%%%%%%%%%%%%%%%%%%%%%%%%%%%%%%%%%%%%%%%%%%%%%%%%%%%%%%%%%%%%%%%%%%%%%%%%%%%%%%%%%%%%%%%%%%%%%%%%%%%%%%%%%%%%%%%%%%%%%
\section{Global Theory}\label{S4}
%%%%%%%%%%%%%%%%%%%%%%%%%%%%%%%%%%%%%%%%%%%%%%%%%%%%%%%%%%%%%%%%%%%%%%%%%%%%%%%%%%%%%%%%%%%%%%%%%%%%%%%%%%%%%%%%%%%%%%%%%

In this section, we establish the global existence of solutions to { the Cauchy problem for} \eqref{equ} in the energy space $C_T(H^1)$. We shall { take advantage of some ideas in} \cite[Section 3]{XZ}. Let us start with { the} global existence result of the non-perturbed problem associated to \eqref{equ}. { Here the underlying energy functional is given by 
$$
E^c(u)=\frac12\int_{\R^3} |\nabla u|^2 \,dx-\frac{1}{6-2b}\int_{\R^3}|x|^{-b}|u|^{6-2b}\,dx.
$$
}
{
\begin{lem}\label{lm1}
Let $0<b<\frac43$, $0<\varepsilon<1$ and $u_0\in H^1_{rad}(\R^3)$ satisfy 
\begin{align}
\|\nabla u_0\|< \|\nabla Q\|, \quad E^c(u_0)<(1-\varepsilon)E^c(Q). \label{as4}
\end{align}
Then there is a unique global solution to the non-perturbed problem associated to \eqref{equ} denoted by $v\in C(\R,H^1_{rad}(\R^3))$, namely
\begin{align} \label{equ0}
\left\{
\begin{aligned}
&\textnormal{i} \partial_t v +\Delta v= - |x|^{-b} |v|^{4-2b}v, \\
&v(0)=u_0. 
\end{aligned}
\right.
\end{align}
Moreover, there holds that
\begin{align} \label{3.2'}
\|\left\langle \nabla\right\rangle v\|_{\Lambda(\R)} \lesssim C(\varepsilon, \|u_0\|),
\end{align}
where 

$|\left\langle \nabla\right\rangle v|:=|v|+|\nabla v|.$
\end{lem}
}
\begin{proof}
{ By \cite[Theorem 1 and Proposition 1]{chl}, then there exists a unique global solution $v\in C(\R,\dot H^1_{rad}(\R^3))$ to \eqref{equ0} satisfying 
$\|\nabla v\|_{\Lambda(\R)}\lesssim_\varepsilon 1.$ 
}
{ In addition}, by { the integral Duhamel formula, the} Strichartz estimates { in Lemma \ref{str}, H\"older's inequality and Hardy's inequality} as stated in Lemma \ref{ckn}, one writes that
\begin{align} \label{3.4}
\begin{split}
\| v\|_{\Lambda(\R)}
&\lesssim \|u_0\|+\|(|x|^{-1}v)^b|v|^{4-3b}v\|_{L^2(\R,L^{6/5})}\\
&\lesssim \|u_0\|+\||x|^{-1}v\|_{W_b(\R)}^b\|v\|_{S(\R)}^{4-3b}\|v\|_{W_b(\R)}\\
&\lesssim \|u_0\|+\|\nabla v\|_{W_b(\R)}^b\|v\|_{S(\R)}^{4-3b}\|v\|_{\Lambda(\R)},
\end{split}
\end{align}
{ where we used H\"older's inequality via the inequalities
\begin{align*}
\frac12=\frac{b}{\frac{10(1+b)}{1+3b}}+\frac{4-3b}{10}+\frac{1}{\frac{10(1+b)}{1+3b}},\\
\frac{5}{6}=\frac{b}{\frac{30(1+b)}{13+9b}}+\frac{4-3b}{10}+\frac{1}{\frac{30(1+b)}{13+9b}}.
\end{align*}
}
Then the proof { is completed} by applying \eqref{3.4} and the property that { $v\in S(\R)$ and $\nabla v \in W_b(\R)$} proved { in \cite[Proposition 1.1]{chl}.
}
\end{proof}

{ Now, we give a good local well-posedness of solutions to the Cauchy problem for \eqref{equ}. It plays an important role in the study of the global well-posedness and scattering of solutions.}

\begin{lem}\label{prt}
{ Let $0<a<1$, $0<b<\frac43$, $0<\varepsilon<1$, and $2<p<6-2a$. Let $u_0 \in H^1_{rad}(\R^3)$ satisfy \eqref{as4}. Then there exist a constant $0<T<<1$ and a unique solution $u \in C([0, T), H^1_{rad}(\R^3))$ to \eqref{equ} with $u(0)=u_0 \in H^1_{rad}(\R^3)$} satisfying
\begin{align*} 
\|\left\langle \nabla\right\rangle u\|_{\Lambda(0,T)}
&\leq C(\varepsilon, \|u_0\|).
\end{align*}
\end{lem}

\begin{proof}
{ To prove the above lemma, we are going to look for a local solution to \eqref{equ} of the form $u:=v+w$, where $v\in C(\R,H^1_{rad}(\R^3))$ is the solution to \eqref{equ0} given by Lemma \ref{lm1} and $w \in C([0, T), H^1_{rad}(\R^3))$ satisfies the equation
\begin{align} \label{3.6}
\left\{
\begin{aligned}
&\textnormal{i}w_t+\Delta w=|x|^{-a}|v+w|^{p-2}(v+w)-|x|^{-b}\left(|v+w|^{4-2b}(v+w)-|v|^{4-2b}v\right),\\ 
&w(0)=0,
\end{aligned}
\right.
\end{align}
where $T>0$ is a small constant to be specified later.
Let $0<\delta<<1$ be a small constant.} In view of \eqref{3.2'}, one { then} splits $[0,T]$ for some $M:=M(\delta, \varepsilon,\|u_0\|)$ as follows, {
\begin{align}
[0,T]:=\bigcup_{0\leq k\leq M}[t_k,t_{k+1}]:=\bigcup_{0\leq k\leq M}I_k, 
\end{align}
\begin{align} \label{3.7}
\|\left\langle \nabla\right\rangle v\|_{\Lambda(I_k)} <\delta,\quad\mbox{for all} \,\,\, 0\leq k\leq M. 
\end{align}
}
Let us resolve \eqref{3.6} by inductive arguments. Assume that there is a solution to \eqref{3.6} on $I_{k-1}$ such that, { for some constants $C, \nu>0$}, 
\begin{align} \label{3.9}
\|\left\langle \nabla\right\rangle w\|_{\Lambda(I_{k-1})}<(2C)^{k-1}T^\nu.
\end{align}
{  We are going to consider the problem on $I_k$. We now introduce} the integral formula, a metric space and a complete norm as follows, {
\begin{align*}
f(w)&:=e^{\textnormal{i}(t-t_k)\Delta}w(t_k)-\textnormal{i}\int_{t_k}^t e^{\textnormal{i}(\cdot-s)\Delta}\left(|x|^{-a}|v+w|^{p-2}(v+w) \right.\\
&\quad \left.-|x|^{-b}\left(|v+w|^{4-2b}(v+w)-|v|^{4-2b}v\right)\right)\,ds, 
\end{align*}
$$
X_k:=\left\{w\in C(I_k,H^1_{rad}(\R^3)) :\quad \|\left\langle \nabla\right\rangle w\|_{\Lambda(I_{k})}<(2C)^kT^\nu\right\}, 
$$

$$
d(w,w'):=\|w-w'\|_{\Lambda(I_k)}. 
$$
}
{ Taking into account} the Strichartz estimates { in Lemma \ref{str}}, one writes that {
\begin{align} \label{3.10}
\begin{split}
\|\left\langle \nabla\right\rangle f(w)\|_{\Lambda(I_k)}
& \lesssim \|w(t_k)\|_{H^1}+\left\|\left\langle \nabla\right\rangle\left(|x|^{-a}|v+w|^{p-2}(v+w)\right)\right\|_{\Lambda'(I_k)}\\
& \quad +\left\|\left\langle \nabla\right\rangle\left(|x|^{-b}\left(|v+w|^{4-2b}(v+w)-|v|^{4-2b}v\right)\right)\right\|_{\Lambda'(I_k)}\\
& \lesssim \|w(t_k)\|_{H^1}+\left\||x|^{-a}|v+w|^{p-2}\left\langle \nabla\right\rangle(v+w)\right\|_{\Lambda'(I_k)}\\
&\quad+\left\||x|^{-a-1}|v+w|^{p-2}(v+w)\right\|_{\Lambda'(I_k)}\\
& \quad+\left\||x|^{-b}\left\langle \nabla\right\rangle \left(|v+w|^{4-2b}(v+w)-|v|^{4-2b}v\right)\right\|_{\Lambda'(I_k)} \\
& \quad + \left\||x|^{-b-1}\left(|v+w|^{4-2b}(v+w)-|v|^{4-2b}v\right)\right\|_{\Lambda'(I_k)}.
\end{split}
\end{align}
Making use of \cite[Lemma 3.4]{guz}, one has that
\begin{align} \label{3.12}
\begin{split}
&\left\||x|^{-a}|v+w|^{p-2}\left\langle \nabla\right\rangle(v+w)\right\|_{\Lambda'(I_k)}+\left\||x|^{-a-1}|v+w|^{p-2}(v+w)\right\|_{\Lambda'(I_k)}\\
&\lesssim T^{\nu}\|\nabla(v+w)\|_{\Lambda(I_k)}^{p-2}\|\left\langle \nabla\right\rangle(v+w)\|_{\Lambda(I_k)}.
\end{split}
\end{align}
Taking advantage of the spirit in the proof of \cite[Proposition 2.1]{chl}, one gets that
\begin{align} \nonumber
&\left\||x|^{-b}\left\langle \nabla\right\rangle \left(|v+w|^{4-2b}(v+w)-|v|^{4-2b}v\right)\right\|_{\Lambda'(I_k)} + \left\||x|^{-b-1}\left(|v+w|^{4-2b}(v+w)-|v|^{4-2b}v\right)\right\|_{\Lambda'(I_k)}\\ 
&\lesssim \left(\|\left\langle \nabla\right\rangle v\|_{\Lambda(I_k)}^{4-2b}+\|\left\langle \nabla\right\rangle w\|_{\Lambda(I_k)}^{4-2b}+\|\left\langle \nabla\right\rangle v\|_{\Lambda(I_k)}^{3-2b}\|\left\langle \nabla\right\rangle w\|_{\Lambda(I_k)}\right)\|\left\langle \nabla\right\rangle w\|_{\Lambda(I_k)}. \label{3.11}
\end{align}
Combining \eqref{3.7}, \eqref{3.9}, \eqref{3.10}, \eqref{3.11} and \eqref{3.12}, we then conclude that, 
\begin{align*}
\|\left\langle \nabla\right\rangle f(w)\|_{\Lambda(I_k)}
&\lesssim \|w(t_k)\|_{H^1}+ T^{\nu}\left(\delta+\|\left\langle \nabla\right\rangle w\|_{\Lambda(I_k)}^{p-2}\right)\left(\delta+\|\left\langle \nabla\right\rangle w\|_{\Lambda(I_k)}\right)\\
&\quad + \left(\delta^{4-2b}+\|\left\langle \nabla\right\rangle w\|_{\Lambda(I_k)}^{4-2b}+\delta^{3-2b}\|\left\langle \nabla\right\rangle w\|_{\Lambda(I_k)}\right)\|\left\langle \nabla\right\rangle w\|_{\Lambda(I_k)}\\
&\leq C(2C)^{k-1}T^\nu+CT^{\nu}\left(\delta+((2C)^kT^{\nu})^{p-2}\right)\left(\delta+(2C)^kT^\nu\right)\\
&\quad + C\left(\delta^{4-2b}+((2C)^kT^\nu)^{4-2b}+\delta^{3-2b}(2C)^kT^\nu\right)(2C)^kT^\nu\\
&< (2C)^{k}T^\nu,
\end{align*}
where we picked $0<\delta, T<<1$ in the last line}. By the fact $M=M(\varepsilon,\delta, \|u_0\|)$, we can choose $T$ uniformly of the process of { the} induction. { Moreover, proceeding as the calculus in \eqref{3.10}, we similarly have that, for any $w,w'\in X_k$, 
\begin{align}
d(f(w),f(w')) \leq \theta d(w,w'), 
\end{align}
}
{ where $0<\theta<1$. At this point, applying the Picard fixed point theorem, then there is a unique solution $w\in C(I_k,H^1_{rad}(\R^3))$ to \eqref{3.6} satisfying 
$$
\|\left\langle \nabla\right\rangle w\|_{\Lambda(I_{k})}<(2C)^{k}T^\nu.
$$ 
Therefore, there is a unique solution $w\in C([0, T),H^1_{rad}(\R^3))$ to \eqref{3.6} such that
\begin{align}
\|\left\langle \nabla\right\rangle w\|_{\Lambda(0,T)}\leq \sum_{k=1}^M\|\left\langle \nabla\right\rangle w\|_{\Lambda(I_k)} \leq \sum_{k=1}^M(2C)^{k}T^\nu \leq M(2C)^{M}T^\nu \leq C(\varepsilon,\delta, \|u_0\|).\label{3.16}
\end{align}
Then $u:= v+w \in C([0, T), H^1_{rad}(\R^3))$ is the solution to the Cauchy problem for \eqref{equ}. By \eqref{3.2'} and \eqref{3.16}, one gets that
\begin{align}
\|\left\langle \nabla\right\rangle u\|_{\Lambda(0,T)}
&\leq \|\left\langle \nabla\right\rangle w\|_{\Lambda(0,T)}+\|\left\langle \nabla\right\rangle v\|_{\Lambda(0,T)} \leq C(\varepsilon, \|u_0\|).\label{3.17}
\end{align}
}
This ends the proof.
\end{proof}

{ Next, we present a coercivity property of solutions to the Cauchy problem for \eqref{equ}}.

{
\begin{lem}\label{crcv} 
Let $0<a<1$, $0<b<\frac43$, $2<p<6-2b$, $0<\varepsilon<1$ and $u_0\in H^1_{rad}(\R^3)$ satisfy 
\begin{align} \label{as5}
\|\nabla u_0\|< \|\nabla Q\|, \quad E(u_0)<(1-\varepsilon)E^c(Q). 
\end{align}
Then there exists $0<\widetilde{\varepsilon}<1$ such that the maximal solution to the Cauchy problem for \eqref{equ} with $u(0)=u_0$ denoted by $u\in C([0, T_{max}), H^1_{rad}(\R^3))$ satisfies
\begin{align}
\|\nabla u\|^2_{{L^\infty}([0, T{_{max}),L^2)}}\leq (1-\widetilde{\varepsilon})\|\nabla Q\|^2.\label{3.18}
\end{align}
Moreover, there holds that $T_{max}=+\infty$ and for any admissible pair $(q,r)\in\Lambda$ and any compact interval $I\subset\R$, we have that
\begin{align} \label{3.19}
\|\left\langle \nabla\right\rangle u\|_{L^q(I,L^r)}\leq (1+|I|^\frac1q)C(\varepsilon,\|u_0\|).
\end{align}
\end{lem}
}
\begin{proof} 
{ By \eqref{defq}, \eqref{as5} and the energy conservation, we write that
\begin{align}\label{3.26} \nonumber
\|\nabla u\|^2-\frac{C_*^{6-2b}}{3-b}\|\nabla u\|^{6-2b} \leq 2E(u)=2E(u_0)&<(1-\varepsilon)\left(\|\nabla Q\|^2-\frac{1}{3-b}\int_{\R^3}|x|^{-b}|Q|^{6-2b}\,dx\right)\\
&=\frac{(1-\varepsilon)(2-b)}{3-b}\|\nabla Q\|^2.
\end{align}
Taking $y:=\frac{\|\nabla u\|^2}{\|\nabla Q\|^2}$ and using \eqref{min'} and \eqref{3.26} give that
\begin{align} \label{3.27}
\begin{split}
\frac{3-b}{2-b}\left(y-\frac{1}{3-b}y^{3-b}\right)
&=\frac{3-b}{2-b}\left(y-\frac{C_*^{6-2b}}{3-b}\|\nabla Q\|^{4-2b}y^{3-b}\right)\\
&=\frac{3-b}{2-b}\left(\frac{\|\nabla u\|^2}{\|\nabla Q\|^2}-\frac{C_*^{6-2b}}{3-b}\frac{\|\nabla u\|^{6-2p}}{\|\nabla Q\|^2}\right)\\
&<(1-\varepsilon).
\end{split}
\end{align}
Therefore, \eqref{3.18} follows from continuity arguments by using \eqref{3.27} and the fact that $\|\nabla u_0\|< \|\nabla Q\|$. Furthermore, we have that $T_{max}=+\infty$ by \eqref{3.18}.} { As a consequence of Lemma \ref{prt}, we then get \eqref{3.19}}. This completes the proof.
\end{proof}

%%%%%%%%%%%%%%%%%%%%%%%%%%%%%%%%%%%%%%%%%%%%%%%%%%%%%%%%%%%%%%%%%%%%%%%%%%%%%%%%%%%%%%%%%%%%%%%%%%%%%%%%%%%%%%%%%%%%%%%%%%%%%%%%%%%%
%%%%%%%%%%%%%%%%%%%%%%%%%%%%%%%%%%%%%%%%%%%%%%%%%%%%%%%%%%%%%%%%%%%%%%%%%%%%%%%%%%%%%%%%%%%%%%%%%%%%%%%%%%%%%%%%%%%%%%%%%%%%%%%%
\section{Energy Scattering }\label{S5}
%%%%%%%%%%%%%%%%%%%%%%%%%%%%%%%%%%%%%%%%%%%%%%%%%%%%%%%%%%%%%%%%%%%%%%%%%%%%%%%%%%%%%%%%%%%%%%%%%%%%%%%%%%%%%%%%%%%%%%%%%
{ In this section, we are going to discuss the energy scattering of solutions to \eqref{equ} and establish Theorem \ref{scattering}.} We shall assume that $0<a<1$, $0<b<\frac43$ and $2<p<6-2b$. In the sprit of Lemma \ref{crcv}, { we denote by $u\in C(\R,H^1_{rad}(\R^3))$ the global solution to the Cauchy problem for \eqref{equ} with $u(0)=u_0 \in H^1_{rad}(\R^3)$ satisfying \eqref{as5} such that
\begin{align}\label{3.18'}
\|\nabla u\|_{L^\infty(\R,L^2)}^2\leq (1-\widetilde{\varepsilon})\|\nabla Q\|^2,
\end{align}
where $0<\widetilde{\varepsilon}<1$.
}
%%%%%%%%%%%%%%%%%%%%%%%%%%%%%%%%%%%%%%%%%%%%%%%%%%%%%%%%%%%%%%%%%%%%%
\subsection{Virial/Morawetz Estimate}\label{vrll}
%%%%%%%%%%%%%%%%%%%%%%%%%%%%%%%%%%%%%%%%%%%%%%%%%%%%%%%%%%%%%%%%%%%%%%%%%%

Let $\psi:\R^3\to\R$ be { a smooth function}. Define the variance potential { by}
\begin{equation}\label{vrl}
V_\psi[u]:=\int_{\R^3}\psi(x)|u(\cdot,x)|^2\,dx.
\end{equation}
{ Define} the Morawetz action { by}
\begin{equation}\label{mrwtz}
M_\psi[u]:=2\Im\int_{\R^3} \overline{u}(\nabla\psi\cdot\nabla u)\,dx. 
\end{equation}
\begin{lem}\label{mrwz1}(\cite[Lemma 4.5]{cc}) {
Let $u\in C(\R, H^1_{rad}(\R^3))$ be the global solution to the Cauchy problem for \eqref{equ}. Then the following identity holds true for any $t \in \R$,
\begin{align*}
V_\psi''[u]=M_\psi'[u]
&=4 \Re \sum_{k,l=1}^3\int_{\R^3}\partial_l\partial_k\psi \partial_ku\partial_l\bar u\,dx-\int_{\R^3}\Delta^2\psi|u|^2\,dx\\
&\quad +\frac{2(p-2)}{p}\int_{\R^3}\Delta\psi|x|^{-a}|u|^{p}\,dx-\frac{4}{p}\int_{\R^3}\nabla\psi\cdot\nabla(|x|^{-a})|u|^{p}\,dx\\
& \quad -\frac{4-2b}{3-b}\int_{\R^3}\Delta\psi|x|^{-b}|u|^{6-2b}\,dx+\frac{2}{3-b}\int_{\R^3}\nabla\psi\cdot\nabla(|x|^{-b})|u|^{6-2b}\,dx.
\end{align*}
}
\end{lem}

{ 
\begin{lem} \label{coer1}
Let $u\in C(\R,H^1_{rad}(\R^3))$ be the global solution to the Cauchy problem for \eqref{equ} satisfying \eqref{3.18'}. Then there holds that
\begin{align*}
\|\nabla u\|^2-\int_{\R^3}|x|^{-b}|u|^{6-2b}\,dx\gtrsim \int_{\R^3}|x|^{-b}|u|^{6-2b}\,dx.
\end{align*}
\end{lem}
\begin{proof}

Taking into account \eqref{min2}, \eqref{min'} and \eqref{3.18'}, we get that
\begin{align} \label{4.17'}
\begin{split}
\|\nabla u\|^2-\int_{\R^3}|x|^{-b}|u|^{6-2b}\,dx
&\geq \left(1-C_*^{6-2b}\|\nabla u\|^{4-2b}\right) \|\nabla u\|^2  \\
& \geq \left(1-C_*^{6-2b}(1-\widetilde{\varepsilon})^{2-b}\|\nabla Q\|^{4-2b}\right)\|\nabla u\|^2 \\
&=\left(1-(1-\widetilde{\varepsilon})^{2-b}\right)\|\nabla u\|^2.
\end{split}
\end{align}
It follows that
\begin{align} \label{4.17''}
\|\nabla u\|^2 \geq \frac{1}{(1-\widetilde{\varepsilon})^{2-b}} \int_{\R^3}|x|^{-b}|u|^{6-2b}\,dx.
\end{align}
By \eqref{4.17'} and \eqref{4.17''}, we then get that
\begin{align}
\|\nabla u\|^2-\int_{\R^3}|x|^{-b}|u|^{6-2b}\,dx\geq\frac{1-(1-\widetilde{\varepsilon})^{2-b}}{(1-\widetilde{\varepsilon})^{2-b} }\int_{\R^3}|x|^{-b}|u|^{6-2b}\,dx. 
\end{align}
This ends the proof. 
\end{proof}

\begin{lem} \label{coer}
Let $u\in C(\R,H^1_{rad}(\R^3))$ be the global solution to the Cauchy problem for \eqref{equ} satisfying \eqref{3.18'}. Let $\psi : \R^3 \to \R$ be a radial smooth function such that $\psi=1$ on $B(1/4)$, $\psi=0$ on $B^c(1/2)$ and $0 \leq \psi \leq 1$. Define $\psi_R:=\psi(\cdot/R)$. Then there holds that, for $R>>1$,
\begin{align*}
\|\nabla (\psi_R u)\|^2-\int_{\R^3} |x|^{-b} |\psi_R u|^{6-2b} \,dx \gtrsim \int_{\R^3} |x|^{-b} |\psi_R u|^{6-2b} \,dx.
\end{align*}
\end{lem}
\begin{proof}
Arguing as the proof of Lemma \ref{coer1}, we only need to show that there exists $\varepsilon'>0$ such that
\begin{align} \label{cc1}
\|\nabla (\psi_R u)\|_{L^{\infty}(\R, L^2)}^2 \leq (1-\varepsilon') \|\nabla Q\|^2.
\end{align}
Observe that
\begin{align*}
\|\nabla (\psi_R u)\|^2&=\|\psi_R\nabla u\|^2-\int_{\R^3}\psi_R\Delta\psi_R|u|^2\,dx\\
& \leq \|\nabla u\|^2 + R^{-2} \|u\|^2,
\end{align*}
where we used integration by parts and the properties of the function $\psi_R$. This together with \eqref{3.18'} gives rise to the desirable conclusion \eqref{cc1} and the proof is completed.
\end{proof}
}
 {
\begin{lem}\label{vrl-mrw}
Let $u\in C(\R,H^1_{rad}(\R^3))$ be the global solution to the Cauchy problem for \eqref{equ} satisfying \eqref{3.18'}. Then there exist $R:=R(Q,\|u_0\|)>>1$ and $\gamma>0$ such that, for any $T>0$,
\begin{equation}\label{mor}
\frac 1T \int_0^T\int_{B(R/4)}|x|^{-b}|u|^{6-2b}\,dx\,dt\lesssim \frac{R}T+R^{-\gamma}.
\end{equation}
Moreover, there exist $\{t_n\}, \{R_n\} \subset \R$ satisfying $t_n, R_n \to\infty$ as $n \to \infty$ such that
\begin{align}\label{2}
\int_{B(R_n)}|x|^{-b}|u(t_n)|^{6-2b} \,dx \to 0  \quad  \mbox{as} \quad n\to\infty.
\end{align} 
\end{lem}
}
\begin{proof}

{
For $R>>1$, we define a radial function $\zeta: \R^3 \to \R$ by
$$
\zeta(x):=\left\{
\begin{aligned}
\frac12|x|^2, &\quad\mbox{if}\quad |x|\leq R/2,\\
R|x|, &\quad\mbox{if}\quad |x|>R.
\end{aligned}
\right.
$$
In addition, we assume that in the centered annulus $\{x\in\R^3 : R/2<|x|<R\}$,}
\begin{align*} 
\partial_r\zeta > 0,\quad\partial^2_r\zeta \geq 0\quad\mbox{and}\quad |\partial^\alpha \zeta|\leq C_\alpha R|\cdot|^{1-{|\alpha|}},\quad\forall\, |\alpha| \geq 1.
\end{align*}

Observe that on the centered ball of radius $R/2$ there holds that
\begin{align} \label{r2}
\partial_j \partial_k \zeta=\delta_{jk},\quad {\Delta \zeta = 3} \quad\mbox{and}\quad \Delta^2 \zeta = 0.
\end{align}
In view of the radial identity
\begin{equation}\label{''}
\partial_j\partial_k=\left(\frac{\delta_{jk}}r-\frac{x_jx_k}{r^3}\right)\partial_r+\frac{x_jx_k}{r^2}\partial_r^2,
\end{equation}
we { are able to compute} that, for $|x| > R$, 
\begin{align*} {
\partial_j \partial_k \zeta =\frac R{|x|}\left(\delta_{jk}-\frac{x_jx_k}{|x|^2}\right), \quad
\Delta \zeta=\frac{2R}{|x|}, \quad
|\Delta^2 \zeta|\lesssim\frac R{|x|^3}.}
\end{align*}
{  Due to the Cauchy-Schwarz inequality, { the conservation of the mass, the fact that $|\nabla \zeta| \lesssim R$} and \eqref{3.18'}, we have that
\begin{align} \label{m1}
|M_\zeta[u]|=2\left|\Im\int_{\R^3} \overline{u}(\nabla\zeta\cdot\nabla u)\,dx\right|\lesssim R.
\end{align}
Furthermore, taking into account the identity \eqref{''}, we obtain that
\begin{align} \label{4.5}
\begin{split}
\Re\sum_{l,k=1}^3\int_{\R^3}\partial_l\partial_k\zeta \partial_{k}u\partial_{l}\overline{u}\,dx
&=\Re \sum_{l.k=1}^3\int_{\R^3}\left(\left(\frac{\delta_{lk}}r-\frac{x_lx_k}{r^3}\right)\partial_r\zeta +\frac{x_lx_k}{r^2}\partial_r^2\zeta \right)\partial_{k} u\partial_{l}\overline{u}\,dx\\
&=\int_{\R^3}\left(|\nabla u|^2-\frac{|x\cdot\nabla u|^2}{|x|^2}\right)\frac{\partial_r\zeta }{|x|}\,dx
+\int_{\R^3}\frac{|x\cdot\nabla u|^2}{|x|^2}\partial_r^2\zeta \,dx \\
&=\int_{\R^3}|\not\nabla u|^2\frac{\partial_r\zeta }{|x|}\,dx+\int_{\R^3}\frac{|x\cdot\nabla u|^2}{|x|^2}\partial_r^2\zeta\,dx,
\end{split}
\end{align}
where the angular gradient is defined by
$$
\not\nabla :=\nabla -\frac{x\cdot\nabla}{|x|^2}x.
$$
Now, by Lemma \ref{mrwz1}, \eqref{r2} and \eqref{4.5}, one writes that
\begin{align}
M_\zeta'[u]
&=4\left(\int_{B(R/2)} |\nabla u|^2 \,dx+\frac{3(p-2)+2a}{2p}\int_{B(R/2)}|x|^{-a}|u|^{p}\,dx-\int_{B(R/2)}|x|^{-b}|u|^{6-2b}\,dx\right)\nonumber\\
& \quad +4\int_{B^c(R/2)}|\not\nabla u|^2\frac{\partial_r\zeta }{|x|}\,dx+4\int_{B^c(R/2)}\frac{|x\cdot\nabla u|^2}{|x|^2}\partial_r^2\zeta\,dx-\int_{B^c(R/2)}\Delta^2\zeta|u|^2\,dx\nonumber\\
& \quad +\int_{B^c(R/2)}\left(\frac{4a}{p}\frac{\nabla\zeta\cdot x}{|x|^2}+\frac{2(p-2)}{p}\Delta \zeta\right) |x|^{-a}|u|^{p}\,dx\nonumber\\
&\quad -\int_{B^c(R/2)}\left(\frac{2b}{3-b}\frac{\nabla\zeta\cdot x}{|x|^2}+\frac{4-2}{3-b}\Delta \zeta\right) |x|^{-b}|u|^{6-2b}\,dx. \nonumber 
\end{align}
The Sobolev embedding inequality \eqref{rineq}, the properties of $\zeta$, the conservation of the mass and \eqref{3.18'} then imply that
\begin{align} \label{4.8}
\begin{split}
M_\zeta'[u]
&\gtrsim 4\left(\int_{B(R/2)} |\nabla u|^2 \,dx+\frac{3(p-2)+2a}{2p}\int_{B(R/2)}|x|^{-a}|u|^{p}\,dx-\int_{B(R/2)}|x|^{-b}|u|^{6-2b}\,dx\right) \\
&\quad -R^{-a}\int_{B^c(R/2)}|u|^{p}\,dx-R^{-b}\int_{B^c(R/2)}|u|^{6-2b}\,dx-R^{-2} \\
&\gtrsim\|\nabla u\|_{L^2(B(R/2))}^2-\int_{B(R/2)}|x|^{-b}|u|^{6-2b}\,dx-R^{-(p-2)-a}-R^{-(4-b)}-R^{-2}.
\end{split}
\end{align}
%%%%%%%%%%%%%%%%%%%%%%%%%%%
Taking advantage of the properties of $\psi$ and integration by parts yield that
\begin{align} \label{4.9'}
\begin{split}
\|\nabla(\psi_Ru)\|^2
&=\|\psi_R\nabla u\|^2-\int_{\R^3}\psi_R\Delta\psi_R|u|^2\,dx\\
&=\|\nabla u\|_{L^2(B(R/2))}^2-\int_{{ R/4<|x| \leq R/2}}(1-\psi_R^2)|\nabla u|^2\,dx-\int_{\R^3}\psi_R\Delta\psi_R|u|^2\,dx,
\end{split}
\end{align}
where $\psi_R$ is the radial function defined in Lemma \ref{coer}. Therefore, by the properties of $\psi_R$ and the conservation of the mass, \eqref{4.9'} readily implies that
\begin{align} \label{4.9}
\begin{split}
\|\nabla u\|_{L^2(B(R/2))}^2
&=\|\nabla(\psi_R u)\|^2+\int_{{ R/4<|x| \leq R/2}}(1-\psi_R^2)|\nabla u|^2\,dx+\int_{\R^3}\psi_R\Delta\psi_R|u|^2\,dx\\
&\gtrsim \|\nabla(\psi_Ru)\|^2-R^{-2}.
\end{split}
\end{align}
Similarly, by the Sobolev embedding inequality \eqref{rineq} and the conservation of the mass, we are able to derive that
\begin{align} \label{4.10}
\begin{split}
-\int_{B(R/2)}|x|^{-b}|u|^{6-2b}\,dx
&=-\int_{\R^3}|x|^{-b}|\psi_Ru|^{6-2b}\,dx-\int_{{ R/4<|x| \leq R/2}}(1-\psi_R^{6-2b})|x|^{-b}|u|^{6-2b}\,dx\\
&\gtrsim-\int_{\R^3}|x|^{-b}|\psi_Ru|^{6-2b}\,dx- R^{-(4-b)}.\\
\end{split}
\end{align}
}
At this point, using Lemma \ref{coer}, \eqref{4.8}, \eqref{4.9} and \eqref{4.10}, we then get that
{
\begin{align} \label{4.11}
\begin{split}
 M_{\zeta}'[u] 
&\gtrsim\|\nabla(\psi_R u)\|^2-\int_{\R^3}|x|^{-b}|\psi_Ru|^{6-2b}\,dx-R^{-(p-2
)-a}-R^{-(4-b)}-R^{-2} \\
&\gtrsim\int_{\R^3}|x|^{-b}|\psi_Ru|^{6-2b}\,dx-R^{-(p-2)-a}-R^{-2}.
\end{split}
\end{align}
Taking $\gamma:=\min \left\{2, p-2+a\right\}$, integrating \eqref{4.11} in time on $[0, T]$ and applying \eqref{m1}, we then obtain that, for $R>>1$,
\begin{align*}
\frac1T\int_0^T\int_{B(R/4)}|x|^{-b}|u|^{6-2b}\,dx\,dt 
&\leq\frac1T\int_0^T\int_{\R^3}|x|^{-b}|\psi_Ru|^{6-2b}\,dx\,dt \\
&\lesssim \frac{\|M_{\zeta}[u]\|_{L^\infty(0,T)}}T+R^{-\gamma}\\
&\lesssim \frac{R}T+R^{-\gamma}.
\end{align*}
}
Then \eqref{mor} follows immediately. Now, taking $R:=4T^\frac{1}{1+\gamma}$, one gets that
\begin{align} \label{1}
\frac 1T \int_0^T\int_{|x|<T^\frac1{1+\gamma}}|x|^{-b}|u|^{6-2b}\,dx\,dt \lesssim T^{-\frac{\gamma}{1+\gamma}}.
\end{align}
{As a consequence, the mean value theorem and \eqref{1} indicate that there exist $\{t_n\}, \{R_n\} \subset \R$ satisfying $t_n, R_n\to\infty$ as $n \to \infty$ such that
$$
\int_{B(R_n)}|x|^{-b}|u(t_n)|^{6-2b}\,dx \to 0 \,\, \mbox{as} \,\, n\to\infty.
$$
} 
Then \eqref{2} holds true and this closes the proof.
\end{proof}

%%%%%%%%%%%%%%%%%%%%%%%%%%%%%%%%%%%%%%%%%%%%%%%%%%%%%%%%%%%%%%%%%%%%%%%%%%%%%%%%%%%%%%%%%%%%%%%%%%%%%%%%%
%%%%%%%%%%%%%%%%%%%%%%%%%%%%%%%%%%%%%%%%%%%%%%%%%%%%%%%%%%%%%%%%%%%%%%%%%%%%%%%%%%%%%%%%%%%%%%%%%%%%%%%%%
\subsection{Scattering Criteria}
%%%%%%%%%%%%%%%%%%%%%%%%%%%%%%%%%%%%%%%%%%%%%%%%%%%%%%%%%%%%%%%%%%%%%%%%%%%%%%%%%%%%%%%%%%%%%%%%%%

Now, { we prove} a scattering criteria in the spirit of \cite{tao}. { For this, we need additionally assume that $p \geq 4$.}

\begin{lem}\label{crt} {
Let $u\in C(\R,H^1_{rad}(\R^3))$ be a global solution to the Cauchy problem for \eqref{equ}}. Assume that 
\begin{equation}\label{ee} 
0<\sup_{t\geq0}\|u(t)\|_{H^1}:=E<\infty.
\end{equation}
There exist $R,\epsilon>0$ depending on $E,a,p,b$ such that if
\begin{equation}\label{crtr} 
{\liminf_{t\to+\infty}\int_{B(R)}|u(t,x)|^{6}\,dx<\epsilon^{6}},
\end{equation}
then $u$ scatters for positive time. 
\end{lem}
\begin{proof}
{ To attain the desired conclusion, by the standard small data theory, it suffices to prove that}
there exist $T, \alpha>0$ such that
\begin{align}\label{lm}
\|e^{\textnormal{i}(\cdot-T)\Delta}u(T)\|_{\Lambda^1(T,\infty)}\lesssim \epsilon^\alpha.
\end{align}
{ In what follows, the aim is to establish \eqref{lm}. Let $\alpha>0$ and $0<\beta<<1$ to be fixed later. Taking into account the Strichartz estimate in Lemma \ref{str}, we know that there exist $\delta>0$ and $T>\epsilon^{-\beta}$ large enough such that 
\begin{equation}\label{fre}
\|e^{\textnormal{i}\cdot\Delta}u_0\|_{S(T,\infty)}< \epsilon^{\delta}.
\end{equation}
}
Let us now take the time slabs $J_1:=[0,T-\epsilon^{-\beta}]$ and $J_2:=[T-\epsilon^{-\beta},T]$, where $T>>1$. The integral Duhamel formula gives that{
\begin{align}\label{int}
\begin{split}
e^{\textnormal{i}(t-T)\Delta}u(T)
&=e^{\textnormal{i}(t-T)\Delta}\left(e^{\textnormal{i}T\Delta}u_0-\textnormal{i}\int_0^Te^{\textnormal{i}(T-s)\Delta}|x|^{-a}|u|^{p-2}u\,ds+\textnormal{i}\int_0^Te^{\textnormal{i}(T-s)\Delta} |x|^{-b}|u|^{4-2b}u \,ds\right)\\
&:=e^{\textnormal{i}t\Delta}u_0-\textnormal{i}\left(\int_{J_1}+\int_{J_2}\right)\left(e^{\textnormal{i}(t-s)\Delta}\left(\mathcal N_a(u)-\mathcal N_b(u)\right)\right)\,ds\\
&:=e^{\textnormal{i} t\Delta}u_0-\textnormal{i}\mathcal F_1-\textnormal{i}\mathcal F_2.
\end{split}
\end{align}
}

{First we estimate the term $\mathcal F_1$}. By the Gagliardo-Nirenberg inequality \eqref{GN}, \eqref{ee} and the dispersive estimate
\begin{align}
\|e^{\textnormal{i}\cdot\Delta}u\|_r\lesssim |\cdot|^{-3(\frac12-\frac 1r)}\|u\|_{r'},\quad \forall \,\, r\geq2,
\end{align}
{ we get that
\begin{align} \label{4.13}
\begin{split}
\|\mathcal F_1\|_{S(T,\infty)}
&\lesssim \left\|\int_{J_1}|t-s|^{-3(\frac12-\frac1{10})}\|\mathcal N_a(u)-\mathcal N_b(u)\|_{\frac{10}9}\,ds\right\|_{L^{10}(T,\infty)}\\
&\lesssim \left\|\int_{J_1}|t-s|^{-\frac{6}{5}}\left(\||x|^{-a}|u|^{p-1}\|_{\frac{10}9}+\||x|^{-b}|u|^{5-2b}\|_{\frac{10}9}\right)\,ds\right\|_{L^{10}(T,\infty)}\\
&\lesssim \left\|\int_{J_1}|t-s|^{-\frac{6}{5}}\left(E^{p-1}+E^{5-2b}\right)\,ds\right\|_{L^{10}(T,\infty)}\\
&\lesssim \left\|(\cdot-T+\varepsilon^{-\beta})^{-\frac15}\right\|_{L^{10}(T,\infty)}\\
&\lesssim \varepsilon^{\frac{\beta}{10}},
\end{split}
\end{align}
where the conditions for the applications of the Gagliardo-Nirenberg inequality \eqref{GN} are satisfied, because we assumed that $0<a<1$ and $0<b<\frac 43$.

}

{ Next we estimate the term $\mathcal F_2$. Let $\psi_R$ be the radial function defined in Lemma \ref{coer}}. By \eqref{crtr}, then there is $T>>1$ such that
\begin{align} \label{3.34}
\int_{\R^3}\psi_R(x)|u(T,x)|^{6}\,dx<\epsilon^{6}.
\end{align}
{ Moreover, utilizing the identity $\partial_t |u|^6=-6\nabla \cdot (|u|^4 \Im(\overline{u} \nabla u))$, integration by parts, the property of $\psi_R$, H\"older's inequality and \eqref{ee}, we are able to infer that
\begin{align} \label{3.35}
\left|\frac\partial{\partial t}\int_{\R^3}\psi_R|u(t)|^{6}\,dx\right| \lesssim \left| \Im \int_{\R^3} \left(\nabla\psi_R\cdot\nabla u\right)|u|^4\bar u\,dx\right| \lesssim \frac 1 R \|u\|_{L^6}^5 \|u\|_{W^{1,6}}
&\lesssim\frac1R\|u\|_{W^{1,6}}.
\end{align}

Integrating in time on $[t, T]$ for any $t\in J_2$, taking $R>\epsilon^{-(6-\beta)}$ and invoking \eqref{3.34} and \eqref{3.35}, we then get that
\begin{align*}
\int_{\R^3}\psi_R|u(t)|^{6}\,dx
&\leq \int_{\R^3}\psi_R|u(T)|^{6}\,dx+\left|\int_t^T\frac\partial{\partial t}\int_{\R^3}\psi_R|u(s)|^{6}\,dx\,ds\right|\nonumber\\
&\lesssim \epsilon^{6}+\frac{\epsilon^{-\beta}}R\nonumber\\
&\lesssim \epsilon^{6}. 
\end{align*}
It follows that
}
\begin{align} \label{3.37}
{\|\psi_Ru\|_{L^\infty(J_2,L^{6})} \lesssim \epsilon^6}.
\end{align}

Using the Strichartz estimates { in Lemma \ref{str} and the Sobolev embedding inequality in $H^1(\R^3)$, we obtain that}
\begin{align} \label{4.180}
\begin{split}
\|\mathcal F_2\|_{S(T,\infty)}
&\lesssim \||x|^{-a}|u|^{p-2}\nabla u\|_{L^1(J_2,L^2)}+\||x|^{-a-1}|u|^{p-2} u\|_{L^1(J_2,L^2)}\\
& \quad +\||x|^{-b}|u|^{4-2b}\nabla u\|_{L^1(J_2,L^2)}+\||x|^{-b-1}|u|^{4-2b}u\|_{L^1(J_2,L^2)}\\
&:=(I)+(II).
\end{split}
\end{align}
{ In the following, we are going to estimate the terms $(I)$ and $(II)$.} Now, we claim that 
\begin{align}\label{claim2}
\|u\|_{L^4(J_2,L^\infty)}&\lesssim {\epsilon^{-\mu\beta}},\quad\mu>0.
\end{align}
We will defer the proof of \eqref{claim2} for now and proceed with the current proof. The term $(II)$ can be estimated as follows, { which is divided into two parts. As an application of H\"older's inequality and Hardy's inequality, we have that}
\begin{align} \label{4.27}
\begin{split}
(II)_1 &:=\|(|x|^{-1}u)^b|u|^{4-3b}\nabla u\|_{L^1(J_2,L^2(B({R/4})))}+\|(|x|^{-1}u)^{1+b}|u|^{4-3b}\|_{L^1(J_2,L^2(B(({R/4})))} \\
&\leq {\||x|^{-1}u\|_{L^2(J_2,L^{6})}^b}\|u\|_{L^4(J_2,L^\infty)}^{2(1-b)}\|u\|_{L^\infty(J_2,L^{6}(B({R/4})))}^{2-b}\|\nabla u\|_{L^2(J_2,L^{6})}\\
&\quad +{\||x|^{-1}u\|_{L^2(J_2,L^{6})}^{1+b}}\|u\|_{L^4(J_2,L^\infty)}^{2(1-b)}\|u\|_{L^\infty(J_2,L^{6}(B({R/4})))}^{2-b}\\
&\lesssim\|u\|_{L^4(J_2,L^\infty)}^{2(1-b)}\|u\|_{L^\infty(J_2,L^{6}(B({R/4})))}^{2-b}\|\nabla u\|_{L^2(J_2,L^{6})}^{1+b}.
\end{split}
\end{align}

Thus, { for suitable choice of $0<\beta<<1$, \eqref{3.19}, \eqref{3.37}, \eqref{claim2} and \eqref{4.27} imply that there exists $\nu>0$ such that} 
\begin{align} \label{004.30}
(II)_1 &\lesssim { \epsilon^{-2(1-b)\mu\beta + 6(2-b)}\left(1+\eps^{-\frac{\beta}{2}}\right)^{1+b}} \leq \epsilon^{\nu} .
\end{align}
On the other hand, by \eqref{ee}, \eqref{claim2}, the Sobolev embedding inequality { in $H^1(\R^3)$} and H\"older's inequality, we obtain that
\begin{align}  \label{004.27}
\begin{split}
(II)_2&:=\||x|^{-b}|u|^{4-2b}\nabla u\|_{L^1(J_2,L^2(B^c({R/4})))}+\||x|^{-b-1}|u|^{4-2b}u\|_{L^1(J_2,L^2(B^c({R/4})))} \\
&\leq R^{-b}\|u\|_{L^4({ J_2}, L^\infty)}^2\|u\|_{L^\infty({ J_2}, L^{ \sigma})}^{2(1-b)}\|\left\langle \nabla\right\rangle u\|_{L^2(J_2,L^{ \sigma})}\\
&\lesssim R^{-b} { \epsilon^{-2\mu \beta}}\|u\|_{L^\infty(H^1)}^{2(1-b)}\|\left\langle \nabla\right\rangle u\|_{L^2(J_2,L^{ \sigma})} \\
&\lesssim R^{-b} { \epsilon^{-2\mu \beta}}\|\left\langle \nabla\right\rangle u\|_{L^\infty(L^2)}^\lambda\|\left\langle \nabla\right\rangle u\|_{L^{2(1-\lambda)}(J_2,L^6)}^{1-\lambda} \\
&\lesssim R^{-b} { \epsilon^{-2\mu \beta}}\|\left\langle \nabla\right\rangle u\|_{L^2(J_2,L^6)}^{1-\lambda} |J_2|^{\frac{\lambda}{2}} \\
&=R^{-b} { \epsilon^{-\left(2\mu+\frac{\lambda}{2}\right)\beta}\|\left\langle \nabla\right\rangle u\|_{L^2(J_2,L^6)}^{1-\lambda}},
\end{split}
\end{align}
{ where
$$
2 \leq \sigma=2(3-2b)<6, \quad 0<\lambda:=\frac{b}{3-2b}<1, \quad \frac{1}{\sigma}=\frac{\lambda}{2}+ \frac{1-\lambda}{6}.
$$ 
}
Therefore, for suitable choice of $R>>1$ and { $0<\beta<<1$}, from \eqref{004.27}, one similarly gets that 
\begin{align} \label{4.301}
(II)_2 \lesssim\epsilon^{\nu} .
\end{align}
{ Consequently, by \eqref{004.30} and \eqref{4.301}, we arrive at
\begin{align} \label{4.30}
(II) \lesssim\epsilon^{\nu}.
\end{align}
}
%%%%%%
Now, we estimate $(I)$ as follows, { which is also separated into two parts. In light of H\"older's inequality and Hardy's inequality, we have that}
\begin{align}\label{04.27}
\begin{split}
(I)_1&:=\|(|x|^{-1}u)^a|u|^{p-2-a}\nabla u\|_{L^1(J_2, L^2(B({R/4})))}+\|(|x|^{-1}u)^{1+a}|u|^{p-2-a}\|_{L^1(J_2, L^2(B({R/4})))} \\
&\leq\||x|^{-1}u\|_{L^2(J_2,L^{6})}^a\|u\|_{L^q(J_2,L^\infty)}^{p-4}\|u\|_{L^\infty(J_2,L^{6}(B({R/4})))}^{2-a}\|\nabla u\|_{L^2(J_2, L^{6})}\\
&\quad +\||x|^{-1}u\|_{L^2(J_2,L^{6})}^{1+a}\|u\|_{L^q(J_2,L^\infty)}^{p-4}\|u\|_{L^\infty(J_2,L^{6}(B({R/4})))}^{2-a}\\
&\lesssim\|u\|_{L^q(J_2,L^\infty)}^{p-4}\|u\|_{L^\infty(J_2,L^{6}(B({R/4})))}^{2-a}\|\nabla u\|_{L^2(J_2,L^{6})}^{1+a} \\
& {\leq |J_2|^\frac{(4-q)(p-4)}{4q}\|u\|_{L^4(J_2,L^\infty)}^{p-4}\|u\|_{L^\infty(J_2,L^{6}(B(R/4)))}^{2-a}\|\nabla u\|_{L^2(J_2,L^{6})}^{1+a}} \\
&={ \epsilon^{-\frac{\beta (4-q)(p-4)}{4q}}\|u\|_{L^4(J_2,L^\infty)}^{p-4}\|u\|_{L^\infty(J_2,L^{6}(B(R/4)))}^{2-a}\|\nabla u\|_{L^2(J_2,L^{6})}^{1+a}},
\end{split}  
\end{align}
{ where $0\leq q=\frac{2(p-4)}{1-a}<4$. }

{ Furthermore, arguing as the proof of \eqref{004.27}, we are able to derive that
\begin{align} \label{est1}
\begin{split}
(I)_2&:=\||x|^{-a}|u|^{p-2}\nabla u\|_{L^1(J_2, L^2(B^c(R/4)))}+\||x|^{-a-1} |u|^{p-2} u\|_{L^1(J_2, L^2(B^c(R/4)))} \\
& \leq  R^{-a}\|u\|_{L^4(J_2, L^\infty)}^2\|u\|_{L^\infty(J_2, L^{2(p-3)})}^{p-4}\|\left\langle \nabla\right\rangle u\|_{L^2(J_2,L^{2(p-3)})} \\
&\lesssim R^{-a} { \epsilon^{-2\mu \beta}}\|u\|_{L^\infty(H^1)}^{p-4}\|\left\langle \nabla\right\rangle u\|_{L^2(J_2,L^{2(p-3)})} \\
&\lesssim R^{-a} { \epsilon^{-2\mu \beta}}\|\left\langle \nabla\right\rangle u\|_{L^\infty(L^2)}^{\theta}\|\left\langle \nabla\right\rangle u\|_{L^{2(1-\theta)}(J_2,L^6)}^{1-\theta} \\
&\lesssim R^{-a} { \epsilon^{-2\mu \beta}}\|\left\langle \nabla\right\rangle u\|_{L^2(J_2,L^6)}^{1-\theta} |J_2|^{\frac{\theta}{2}} \\
&=R^{-a}\epsilon^{-\left(2\mu+\frac{\theta}{2}\right)\beta}\|\left\langle \nabla\right\rangle u\|_{L^2(J_2,L^6)}^{1-\theta},
\end{split}
\end{align}
where 
$$
0<\theta:=\frac{6-p}{2(p-3)}<1, \quad \frac{1}{2(p-3)}=\frac{\theta}{2}+\frac{1-\theta}{6}.
$$
Therefore, for suitable choice of $R>>1$ and $0<\beta<<1$, by \eqref{04.27} and \eqref{est1}, one analogously has that 
}
\begin{align} \label{14.27}
(I)&\lesssim\epsilon^{\nu}.
\end{align}
%%%%
%
Hence, combining \eqref{4.180}, \eqref{4.30} and \eqref{14.27}, one gets that 
\begin{align} \label{4.31}
\|\mathcal F_2\|_{S(T,\infty)} &\lesssim\epsilon^{\nu}
\end{align}

At this stage, making use of \eqref{fre}, \eqref{4.13} and \eqref{4.31}, we then conclude that \eqref{lm} holds true.

Lastly, we move on to the proof of Claim \eqref{claim2}. Utilizing the integral Duhamel formula, { the Strichatz estimates in Lemma \ref{str}}, \eqref{fretao} and \eqref{ee}, we obtain { that}
\begin{align} \label{4.21}
\begin{split}
\|u\|_{L^4(J_2,L^\infty)}
&\lesssim E+\||x|^{-a}|u|^{p-2}\nabla u\|_{L^1(J_2,L^2)}+\||x|^{-a-1}|u|^{p-1}\|_{L^1(J_2,L^2)}\\
& \quad +\||x|^{-b}|u|^{4-2b}\nabla u\|_{L^1(J_2,L^2)}+\||x|^{-b-1}|u|^{5-2b}\|_{L^1(J_2,L^2)} \\
&:=(III)+(IV).
\end{split}
\end{align}
{ Next we are going to estimate the term $(III)$ and $(IV)$. Thanks to H\"older's inequality and Hardy's inequality, we infer that
\begin{align*} 
&\|(|x|^{-1}u)^b|u|^{4-3b}\nabla u\|+\|(|x|^{-1}u)^{1+b}|u|^{4-3b}\| \\
&\leq\||x|^{-1}u\|_{6}^b\|u\|_{r}^{4-3b}\|\nabla u\|_{6}+\||x|^{-1}u\|_{6}^{1+b}\|u\|_{r}^{4-3b} \\
&\lesssim\|u\|_{r}^{4-3b}\|\nabla u\|_{6}^{1+b}, 
\end{align*}
where 
$$
r:=\frac{6(4-3b)}{2-b}>0.
$$
}{ Furthermore, by the Sobolev embedding inequality, H\"older's inequality and \eqref{ee}, we know that
\begin{align*} 
\|u\|_{r}^{4-3b}\|\nabla u\|_{6}^{1+b} 
&\lesssim\|\nabla u\|_{\widetilde{r}}^{4-3b}\|\nabla u\|_{6}^{1+b} \\
&\lesssim\left(\|\nabla u\|^{\tau}\|\nabla u\|_{6}^{1-\tau}\right)^{4-3b}\|\nabla u\|_{6}^{1+b}\\
&\lesssim\|\nabla u\|_{6}^{2},
\end{align*}
where
$$
\widetilde{r}:=\frac{6(4-3b)}{10-7b}>0, \quad 0<\tau:=\frac{3-2b}{4-3b}<1, \quad  \frac{1}{\widetilde{r}}=\frac{\tau}{2}+\frac{1-\tau}{6}.
$$

It then follows from \eqref{3.19} that
\begin{align} \label{4.24}
(IV) \lesssim \left(1+ \eps^{-\frac{\beta}{2}}\right)^2.
\end{align}
%%%%%%%%%%%%%%%%%%%%%
Similarly, employing H\"older's inequality and Hardy's inequality, one can write that
\begin{align*} 
&\|(|x|^{-1}u)^a|u|^{p-2-a}\nabla u\|+\|(|x|^{-1}u)^{1+a}|u|^{p-2-a}\|\\
&\leq\||x|^{-1}u\|_{6}^a\|u\|_{\rho}^{p-2-a}\|\nabla u\|_{6}+\||x|^{-1}u\|_{6}^{1+a}\|u\|_{\rho}^{p-2-a} \\
&\lesssim\|u\|_{\rho}^{p-2-a}\|\nabla u\|_{6}^{1+a},
\end{align*}
where
$$
\rho:=\frac{6(p-2-a)}{2-a}>0,
$$
Applying the Sobolev embedding inequality, H\"older's inequality and \eqref{ee}, one then gets that
\begin{align*}  
\|u\|_{\rho}^{p-2-a}\|\nabla u\|_{6}^{1+a} 
&\lesssim\|\nabla u\|_{\widetilde{\rho}}^{p-2-a}\|\nabla u\|_{6}^{1+a}\\
&\lesssim\left(\|\nabla u\|^{\kappa}\|\nabla u\|_{6}^{1-\kappa}\right)^{p-2-a}\|\nabla u\|_{6}^{1+a}\\
&\lesssim\|\nabla u\|_{6}^{\gamma},
\end{align*}
where $1<\gamma:=\frac{p-4}{2}+1+a<2$ and
$$
\widetilde{\rho}:=\frac{6(p-2-a)}{2(p-2)+2-3a}, \quad 0<\kappa:=\frac{p-2a}{2(p-2-a)}<1, \quad \frac{1}{\widetilde{\rho}}=\frac{\kappa}{2}+\frac{1-\kappa}{6}.
$$
Thereby, using \eqref{3.19} and H\"older's inequality, we derive that
\begin{align} \label{a4.24}
(III) \lesssim\|\nabla u\|_{L^2(J_2, L^6)}^{\gamma}|J_2|^\frac{2-\gamma}{2} \lesssim \left(1+\eps^{-\frac{\beta}{2}}\right)^{\gamma} \epsilon^{-\frac{\beta(2-\gamma)}{2}}.
\end{align}
%%%%%%%%%%%%%%%%%%%%%%%%
As a consequence, from \eqref{4.24} and \eqref{4.24}, we conclude that there exists $\mu>0$ such that
\begin{align*} 
\|u\|_{L^4(J_2,L^\infty)}&\lesssim \epsilon^{-\mu\beta}.
\end{align*}
}
The assertion \eqref{claim2} is now fully established and the proof is completed.
\end{proof}

%%%%%%%%%%%%%%%%%%%%%%%%%%%%%%%%%%%%%%%%%%%%%%%%%%%%%%%%%%%%%%%%%%%%%%%%%%%%%%%%%%%%%%%%%%%%%%%%%%%%%%%%%%%%%%%
%%%%%%%%%%%%%%%%%%%%%%%%%%%%%%%%%%%%%%%%%%%%%%%%%%%%%%%%%%%%%%%%%%%%%%%%%%%%%%%%%%%%%%%%%%%%%%%%%%
\begin{proof} [Proof of Theorem \ref{scattering}]
{  In light of Lemma \ref{k} and Lemma \ref{crcv}, then the global existence of solutions follows necessarily. As a result of \eqref{2} in Lemma \ref{vrl-mrw}, we find that, for $R_n>R$,
$$
\int_{B(R)}  |u(t_n, x)|^{6-2b}\,dx \leq R^b \int_{B(R_n)} |x|^{-b} |u(t_n, x)|^{6-2b}\,dx=o_n(1).
$$ 
This together with Lemma \ref{k} and Lemma \ref{crcv} 
infers that the scattering conditions outlined in Lemma \ref{crt} hold true.} Then the proof is completed.
\end{proof}
%%%%%%%%%%%%%%%%%%%%%%%%%%%%%%%%%%%%%%%%%%%%%%%%%%%%%%%%%%%%%%%%%%%%%%%%%%%%%%%%%%%%%%%%%%%%%%%%%%%%%%%%%%%%%%%%%%%%%%%%%

\section{Blow-up}
\label{S6}
In this section, we will discuss the blow-up of solutions to the Cauchy problem for \eqref{equ} and provide the proof for Theorems \ref{nonexistence} and \ref{blowup}. Let
$$
K^c(u)=\int_{\R^3} |\nabla u|^2 \,dx -\int_{\R^3} |x|^{-b}|u|^{6-2b} \,dx,
$$
\begin{align*}
I(u)=\frac{3(p-2)-2(2-a)}{6(p-2)+4a}\int_{\R^3} |\nabla u|^2 \,dx + \frac{2(6-2b)-3(p-2)-2a}{(3(p-2)+2a)(6-2b)} \int_{\R^3} |x|^{-b}|u|^{6-2b} \,dx.
\end{align*}

\begin{lem} \label{clem2}
{ Let $0<b<a<2$ and $2+\frac{4-2a}{3}<p < 6-2a$}. Then there holds that
\begin{align*}
m&=\inf \left\{ I(u) : u \in H^1(\R^3) \backslash \{0\}, K^c(u) < 0\right\} \\
&=\inf \left\{ I(u) : u \in H^1(\R^3) \backslash \{0\}, K^c(u) \leq 0\right\}.
\end{align*}
\end{lem}
\begin{proof}
Since $K^c(u) \leq K(u)$ for any $u \in H^1(\R^3)$, by Lemma \ref{clem1}, then
\begin{align} \label{c111-11}
\begin{split}
m&=\inf \left\{ I(u) : u \in H^1(\R^3) \backslash \{0\}, K(u) < 0\right\} \\
& \geq \inf \left\{ I(u) : u \in H^1(\R^3) \backslash \{0\}, K^c(u) < 0\right\}.
\end{split}
\end{align}
Let $u \in H^1(\R^3)$ be such that $K^c(u)<0$. Define 
$$
u^{\lambda}(x)=\lambda^{\frac 12} u(\lambda x), \quad x \in \R^3.
$$
It is not hard to calculate that $K^c(u^{\lambda})=K^c(u)$, ${ I(u^{\lambda})=I(u)}$ and
$$
K(u^{\lambda})=\int_{\R^3} |\nabla u|^2 \,dx -\int_{\R^3} |x|^{-b}|u|^{6-2b}\,dx +\lambda^{\frac p 2 -3 +a} \frac{3(p-2) +a}{2p} \int_{\R^3}|x|^{-a} |u|^{p} \,dx.
$$
Then we are able to infer that $K(u^{\lambda}) \to K^c(u)$ as $\lambda \to +\infty$, because of $p<6-2a$. Therefore, there holds that
\begin{align} \label{c12}
\begin{split}
m&=\inf \left\{ I(u) : u \in H^1(\R^3) \backslash \{0\}, K(u) < 0\right\} \\
&  \leq \inf \left\{ I(u) : u \in H^1(\R^3) \backslash \{0\}, K^c(u) < 0\right\}.
\end{split}
\end{align}
Consequently, combining \eqref{c111-11} and \eqref{c12} results in
\begin{align} \label{c11-1}
\begin{split}
m&=\inf \left\{ I(u) : u \in H^1(\R^3) \backslash \{0\}, K(u) < 0\right\} \\
&  = \inf \left\{ I(u) : u \in H^1(\R^3) \backslash \{0\}, K^c(u) < 0\right\}.
\end{split}
\end{align}
Apparently, there holds that
\begin{align} \label{c13}
\begin{split}
&\inf \left\{ I(u) : u \in H^1(\R^3) \backslash \{0\}, K^c(u) < 0\right\} \\
& \geq \inf \left\{ I(u) : u \in H^1(\R^3) \backslash \{0\}, K^c(u) \leq 0\right\}.
\end{split}
\end{align}
Let $u \in H^1(\R^3)$ be such that $K^c(u) \leq 0$. { Recall that 
$$
u_{\lambda}(x)=\lambda^{\frac 32} u(\lambda x), \quad x \in \R^3.
$$
} 
Observe that 
$$
K^c(u_{\lambda})=\lambda^2\int_{\R^3} |\nabla u|^2 \,dx -\lambda^{6-2b}\int_{\R^3}|x|^{-b}|u|^{6-2b} \,dx,
$$
$$
I(u_{\lambda})=\frac{3(p-2)-2(2-a)}{6(p-2)+4a} \lambda^2 \int_{\R^3} |\nabla u|^2 \,dx + \frac{2(6-2b)-3(p-2)-2a}{(3(p-2)+2a)(6-2b)}\lambda^{6-2b}\int_{\R^3} |x|^{-b}|u|^{6-2b} \,dx.
$$
It then indicates that $K^c(u_{\lambda})<0$ for any $\lambda>1$, { due to $b<2$. In addition, one has that $I(u_{\lambda}) \to I(u)$ as $\lambda \to 1^+$}.  As a result, there holds that
\begin{align*}
&\inf \left\{ I(u) : u \in H^1(\R^3) \backslash \{0\}, K^c(u) < 0\right\} \\
& \leq \inf \left\{ I(u) : u \in H^1(\R^3) \backslash \{0\}, K^c(u) \leq 0\right\}.
\end{align*}
This together with \eqref{c13} leads to 
\begin{align*}
&\inf \left\{ I(u) : u \in H^1(\R^3) \backslash \{0\}, K^c(u) < 0\right\} \\
& = \inf \left\{ I(u) : u \in H^1(\R^3) \backslash \{0\}, K^c(u) \leq 0\right\}.
\end{align*}
Going back to \eqref{c11-1}, we then obtain the desired conclusion and the proof is completed.
\end{proof}

\begin{proof}[Proof of Theorem \ref{nonexistence}] 
It follows from Theorem \ref{existence} that there exists no minimizers to \eqref{minn}, { because any minimizer to \eqref{minn} is a solution to \eqref{equ1} for $\omega=0$.} Let us now prove that \eqref{cm} holds true. Using Lemma \ref{clem2}, we have that
\begin{align} \label{c111}
\begin{split}
m&=\inf \left\{I(u) : u \in H^1(\R^3) \backslash \{0\}, K^c(u) \leq 0\right\} \\
& \geq \inf \left\{I(u) + N(u) : u \in H^1(\R^3) \backslash \{0\}, K^c(u) \leq 0\right\},
\end{split}
\end{align}
where
\begin{align*}
N(u)&:=\frac{2(6-2b)-3(p-2)-2a}{(3(p-2)+2a)(6-2b)} \left(\int_{\R^3} |\nabla u|^2 \,dx - \int_{\R^3} |x|^{-b}|u|^{6-2b} \,dx \right)\\
&=\frac{2(6-2b)-3(p-2)-2a}{(3(p-2)+2a)(6-2b)} K^c(u).
\end{align*}
Clearly, the equality holds in \eqref{c111} if and only if $N(u)=0$, i.e.
$$
\int_{\R^3} |\nabla u|^2 \,dx =\int_{\R^3} |x|^{-b}|u|^{6-2b} \,dx.
$$
Observe that
\begin{align}  \label{bu2} \nonumber
&\inf \left\{I(u) + N(u) : u \in H^1(\R^3) \backslash \{0\}, K^c(u) \leq 0\right\} \\ \nonumber
& = \frac{3(p-2)(2-b)+2(2-b)a}{(6(p-2)+4a)(3-b)}\inf \left\{ \|\nabla u\|^2 : u \in H^1(\R^3) \backslash \{0\}, K^c(u) \leq 0\right\} \\ \nonumber
&=\frac{3(p-2)(2-b)+2(2-b)a}{(6(p-2)+4a)(3-b)} \inf \left\{ \|\nabla u\|^2 \left(\frac{\|\nabla u\|^2}{\int_{\R^3} |x|^{-b}|u|^{6-2b} \,dx}\right)^{1/(2-b)}: u \in H^1(\R^3) \backslash \{0\}\right\} \\ \nonumber
&=\frac{3(p-2)(2-b)+2(2-b)a}{(6(p-2)+4a)(3-b)}\inf \left\{\left(\frac{\|\nabla u\|^{6-2b}}{\int_{\R^3} |x|^{-b}|u|^{6-2b} \,dx}\right)^{1/(2-b)}: u \in \dot{H}^1(\R^3) \backslash \{0\}\right\} \\ \nonumber
&=\frac{3(p-2)(2-b)+2(2-b)a}{(6(p-2)+4a)(3-b)} \frac{1} {{C_*}^{(6-2b)/(2-b)}} \\ 
&=E^c(Q),
\end{align}
where $C_*>0$ is the constant given by \eqref{min2}. This completes the proof.
\end{proof}

\begin{proof}[Proof of Theorem \ref{blowup}]
Let $u \in C([0, T_{max}), H^1(\R^3))$ be the solution to the Cauchy problem for \eqref{equ} with initial datum $u_0 \in \mathcal{K}^-$. First we verify that $u(t) \in \mathcal{K}^-$ for any $t \in [0, T_{max})$. Suppose that there exists $t_0 \in (0, T_{max})$ such that $u(t_0) \not\in \mathcal{K}^-$. By the conservation of laws, we know that $E(u(t))=E(u_0)<m$ for any $t \in [0, T_{max})$. Then there holds that $K(u(t_0)) \geq 0$. Since $K(u_0)<0$, then there exists $t_1 \in (0, t_0]$ such that $K(u(t_1))=0$, which indicates that $m\leq E(u(t_1))$. This is impossible by the conservation of the energy. Hence the desirable result follows. 

Let $K(u)<0$, by Lemma \ref{maximum}, then there exists $0<\lambda_u<1$ such that $K(u_{\lambda_u})=0$ and $E(u_{\lambda_u}) \geq m$. In addition, the function $\lambda \mapsto E(u_{\lambda})$ is concave on $[\lambda_u, +\infty)$. It then follows that
$$
E(u)-E(u_{\lambda_u})=\frac{d}{d\lambda} E(u_{\lambda})\mid_{\lambda=\xi} (1-\lambda_u) \geq K(u)(1-\lambda_u) \geq K(u), \quad \xi \in [\lambda_u, 1].
$$
By the conservation of the energy and the fact that $E(u_{\lambda_u}) \geq m$, then 
$$
K(u(t)) \leq E(u_0)-m<0, \quad \forall \, t \in [0, T_{max}).
$$

To establish blow-up of the solution, we need to introduce the following localized virial quantity defined by
$$
V_R(t):=\int_{\R^3} \psi_R (x) |u(t,x)|^2 \,dx, \quad \psi_R(x):=R^2 \psi \left(\frac{|x|}{R}\right),
$$
where $\psi \in C_0^{\infty}(\R^3)$ is a radial function such that $\psi(r)=r^2$ for $0 \leq r \leq 1$, $\psi(r)=0$ for $r \geq 3$, $\partial_r \psi(r) \leq 2r$ and $\partial_r^2 \psi(r) \leq 2$ for $r \geq 0$. It is easy to notice that $\Delta \psi_R(r)=6$ for $0 \leq r \leq R$ and $\Delta^2 \psi_R(r) =0$ for $0 \leq r \leq R$. In the spirit of Lemma \ref{mrwz1}, then there holds that
\begin{align*}
V''_R(t)&=4 \textnormal{Re} \sum_{j, k=1}^{3} \int_{\R^3} \partial_{r}^2 \psi_R |\nabla u(t)|^2 \,dx -\int_{\R^3} \Delta^2 \psi_R |u(t)|^2 \,dx \\
& \quad  +\frac{2(p-2)}{p} \int_{\R^3} |x|^{-a}|u(t)|^{p} \Delta \psi_R \,dx-\frac{4}{p} \int_{\R^3} \nabla \left(|x|^{-a}\right) \cdot \nabla \psi_R |u(t)|^{p} \,dx\\
& \quad -\frac{4-2b}{3-b} \int_{\R^3} |x|^{-b}|u(t)|^{6-2b} \Delta \psi_R \,dx +\frac{2}{3-b} \int_{\R^3} \nabla \left(|x|^{-b}\right) \cdot \nabla \psi_R |u(t)|^{6-2b} \,dx.
\end{align*}
Further, we can compute that
\begin{align*}
V''_R(t)& \lesssim 8 K(u(t)) + R^{-2}\int_{R \leq |x| \leq 3R} |u(t)|^2 \,dx \\
& \quad + R^{-a} \int_{R \leq |x| \leq 3R} |u(t)|^{p} \,dx +R^{-b} \int_{R \leq |x| \leq 3R} |u(t)|^{6-2b} \,dx.
\end{align*}
Applying Lemma \ref{rlem} and the conservation of the mass, one gets that
$$
R^{-a} \int_{R \leq |x| \leq 3R} |u(t)|^{p} \,dx \lesssim R^{-(p-2)-a} \|\nabla u(t)\|^{\frac{p-2}{2}},
$$
$$
R^{-b} \int_{R \leq |x| \leq 3R} |u(t)|^{6-2b} \,dx \lesssim R^{-(p-4)} \|\nabla u(t)\|^{2-b}.
$$
Using Young's inequality, then, for any $\eps>0$ small, then there exists $R>1$ large enough such that
\begin{align} \label{bu}
\begin{split}
V''_R(t) & \lesssim 8 K(u(t)) + \eps \|\nabla u(t)\|^2 + \eps \\
&= 16(3-b) E(u(t))- (16-8b-\eps) \|\nabla u(t)\|^2 \\
& \quad -\frac{16(3-b)-12(p-2)-8a}{p} \int_{\R^3} |x|^{-a}|u(t)|^p \,dx + \eps \\
& \leq 16(3-b) E(u(t))- (16-8b-\eps) \|\nabla u(t)\|^2 +\eps,
\end{split}
\end{align}
where we used fact that $16(3-b)-12(p-2)-8a>0$ is due to $b<a$ and $p<6-2a$.
Since $K(u(t))<0$ for any $t \in [0, T_{max})$, by \eqref{min2}, Lemma \ref{clem1} and \eqref{bu2}, then we know that
\begin{align*}
m \leq I(u(t)) &\leq \frac{3(p-2)(2-b)+2(2-b)a}{(6(p-2)+4a)(3-b)}\int_{\R^3} |x|^{-b}|u(t)|^{6-2b} \,dx \\
&\leq \frac{3(p-2)(2-b)+2(2-b)a}{(6(p-2)+4a)(3-b)} C_*^{6-2b} \|\nabla u(t)\|^{6-2b} \\
&= \left(\frac{3(p-2)(2-b)+2(2-b)a}{(6(p-2)+4a)(3-b)} \right)^{3-b} m^{b-2} \|\nabla u(t)\|^{6-2b},
\end{align*}
from which we obtain that
\begin{align} \label{bu1}
\|\nabla u(t)\|^2 \geq \frac{(6(p-2)+4a)(3-b)} {3(p-2)(2-b)+2(2-b)a}m.
\end{align}
Since $E(u_0)<m$, by the conservation of the energy, then there exists $\delta_0>0$ such that $E(u(t)) \leq (1-\delta_0) m$ for any $t \in [0, T_{max})$. Taking $\eps>0$ small enough and applying \eqref{bu} and \eqref{bu1}, we then derive that
\begin{align*}
V''_R(t) &\lesssim 16(3-b) (1-\delta_0)m-  \frac{(6(p-2)+4a)(3-b)(16-8b-\eps)}{3(p-2)(2-b)+2(2-b)a}m +\eps \\
& = 16(3-b) (1-\delta_0)m-16(3-b)m + \frac{(6(p-2)+4a)(3-b) \eps}{3(p-2)(2-b)+2(2-b)a}m+\eps \\
& <-8(3-b) \delta_0 m,
\end{align*}
where we used the equality
$$
\frac{(6(p-2)+4a)(16-8b)}{3(p-2)(2-b)+2(2-b)a}=16.
$$
It obviously follows that $u$ blows up in finite time and the proof is completed.
\end{proof}
\begin{comment}
\begin{Acknow}	
 The authors are grateful to ...
\end{Acknow}
\end{comment}
\hrule

\vspace{0.3cm}

{\noindent{\bf\large Declarations.}}
On behalf of all authors, the corresponding author states that there is no conflict of interest. No data-sets were generated or analyzed during the current study.

\vspace{0.3cm}

\hrule 
%%%%%%%%%%%%%%%%%%%%%%%%%%%%%%%%%%%%%%%%%%%%%%%%%%%%%%%

\end{document}